\documentclass[a4paper]{article}
\usepackage{amsmath}
\usepackage{graphicx}
\usepackage[colorinlistoftodos]
{todonotes}
\usepackage{setspace}
\usepackage{titlesec}
\usepackage{enumerate}
\usepackage{amsmath}
\usepackage{amsthm}
\usepackage{fancyvrb}
\usepackage{float}
\usepackage{lmodern}
\usepackage[utf8]{inputenc}
\usepackage[T1]{fontenc}
\usepackage{hyperref}
\hypersetup{bookmarksnumbered=true}
\usepackage{tabulary}
\usepackage{amsfonts}
\usepackage{graphicx}
\usepackage[linesnumbered,ruled,vlined]{algorithm2e}
\usepackage[ngerman,english]{babel}
\usepackage{tikz}
\usepackage{color}
\usepackage{url}
\usepackage{mathrsfs}
\usepackage[mathscr]{euscript}
\usepackage{empheq}
\usepackage{pbox}
\usepackage{csquotes}
\usepackage{afterpage}
\usetikzlibrary{matrix,backgrounds}
\pgfdeclarelayer{myback}
\pgfsetlayers{myback,background,main}

\tikzset{mycolor/.style = {line width=1bp,color=#1}}%
\tikzset{myfillcolor/.style = {draw,fill=#1}}%

\newtheorem{thrm}{Theorem}

\newtheorem{conj}{Conjecture}
\newtheorem{prop}{Proposition}
\newtheorem{obs}{Observation}
\newtheorem{cor}{Corollary}
\newtheorem{quest}{Question}{\bfseries}{\rmfamily}
\newtheorem{definition}{Definition}

\newcommand{\gen}[1]{\ensuremath{\langle #1\rangle}}

\usepackage{authblk}
\title{Cutting Planes for Families Implying Frankl's Conjecture}
\author[1,2]{Jonad Pulaj \thanks{The work for this article has been (partly) conducted within the Research Campus MODAL funded by the German Federal Ministry of Education and Research (BMBF grant number 05M14ZAM).}}
\affil[1]{Department of Computer Science, Mathematics and Physics, Faculty of Science and Technology, The University of the West Indies, Cave Hill, St. Michael, Barbados}
\affil[2]{Dept. of Mathematical Optimization, Zuse Institute Berlin (ZIB)
Takustr. 7, 14195 Berlin, Germany}
\date{}                     %% if you don't need date to appear
\affil[ ]{\tt jonad.pulaj@cavehill.uwi.edu}
\begin{document}
\maketitle
\begin{abstract}
We find previously unknown families of sets which ensure Frankl's conjecture  holds for all families that contain them using an algorithmic framework. The conjecture states that for any nonempty union-closed (UC) family there exists an element of the ground set in at least half the sets of the considered UC family. Poonen's Theorem characterizes the existence of weights which determine whether a given UC family implies the conjecture for all UC families which contain it. We design a cutting-plane method that computes the explicit weights which imply the existence conditions of Poonen's Theorem. This method enables us to answer several open questions regarding structural properties of UC families, including the construction of a counterexample to a conjecture of Morris from 2006.
\end{abstract}
\section{Introduction}
Frankl's (union-closed sets) conjecture is a celebrated unsolved problem in combinatorics that was recently brought to the attention of a wider audience as a polymath project led by Timothy Gowers \cite{Gowers}. A nonempty finite family of distinct finite sets $\mathcal{F}$ is union-closed (UC) if and only if for every $A,B \in \mathcal{F}$ it follows that $A \cup B \in \mathcal{F}$. Frankl's conjecture states that for any UC family $\mathcal{F} \neq \left\{\emptyset\right\}$ there exists an element in the union of sets of $\mathcal{F}$ that is present in at least half the sets of $\mathcal{F}$. The problem appears to have little structure\textemdash perhaps the very reason why a proof or disproof remains elusive. 

In this paper we focus on a well-established method employed to attack the problem referred to as \emph{local configurations} in Bruhn and Schaudt \cite{Survey}, namely UC families that imply the conjecture for all UC families which contain them. In other words, these particular UC families always have an element in their sets that is frequent enough to imply the conjecture for all UC families that contain them. In this regard, given a UC family $\mathcal{A}$, Poonen's Theorem  \cite{Poonen} characterizes the necessity of the implication by the existence of weights on the elements of $\mathcal{A}$ that obey certain inequalities. Following Vaughan \cite{Vaughan1}, we say that a UC family of sets $\mathcal{A}$ with a largest (cardinality-wise) set $A$ is \emph{Frankl-Complete} (FC), if and only if for every UC family $\mathcal{F} \supseteq \mathcal{A}$ there exists $i \in A$ that is contained in at least half the sets of $\mathcal{F}$. A UC family $\mathcal{A}$ with a largest (cardinality-wise) set $A$ is \emph{Non\textendash Frankl-Complete} (Non\textendash FC), if and only if there exists a UC family $\mathcal{F} \supseteq \mathcal{A}$ such that each $i \in A$ is in less than half the sets of $\mathcal{F}$. Non\textendash FC-families are particularly useful in characterizing \emph{minimal} FC-families, i.e., FC-families that do not contain smaller FC-families, and also other objects of interests defined in Morris \cite{Morris}, which help shed light into structural properties of the conjecture. In addition, Non\textendash FC-families yield natural candidates for possible counterexamples. 

However, on a more positive note, the pressing relevance of FC and Non\textendash FC-families is evident in existing literature: These objects are at the heart of arguments that yield improved bounds for the problem, as seen in Poonen \cite{Poonen}, Gao and Yu \cite{Gao}, Morris \cite{Morris}, Markovi\'c \cite{Mark}, Bo\v snjak and Markovi\'c \cite{11case}, and finally Vu\v{c}kovi\'c and \v{Z}ivkovi\'c \cite{12case} which features the current best bound of $n \leq 12$, where $n$ is the cardinality of the largest set in a UC family. In other words, Frankl's conjecture holds for all UC families whose largest set has at most twelve elements. Furthermore, FC-families are used in Bruhn et al. \cite{Bruhn} to prove that Frankl's conjecture holds for subcubic bipartite graphs. Therefore characterizing a considerable number of previously unknown FC and Non\textendash FC-families\textemdash the fundamental contribution of this work which consequently helps settle several open questions of interest\textemdash is a clear step toward a better understanding of Frankl's conjecture. 

Characterizing \emph{exactly} which UC families are FC and Non\textendash FC is surprisingly difficult, as evinced by the relative dearth of known FC-families despite the past twenty-five years of research on the matter. For a positive integer $r$ an $r$-set (or $r$-subset) is a set (or subset) of cardinality $r$. Previous researchers use special structures and stronger than necessary conditions to determine a number of FC-families. In particular, Poonen \cite{Poonen} proves that any UC family which contains three 3-subsets of a 4-set satisfies the conjecture. Vaughan \cite{Vaughan1}, \cite{Vaughan2}, \cite{Vaughan3} proves that the conjecture holds for any UC family which contains a 5-set and all of its 4-subsets, or ten of the 4-subsets of a 6-set, or three 3-subsets of a 7-set with a common element. Furthermore, using a heuristic procedure implemented in a computer algebra system, Vaughan identifies potential weight systems for candidate FC-families and then proves through tedious and technical case analysis that a few more UC families are FC. Still, several FC-families Vaughan discovers are not minimal, in the sense that they contain smaller FC-families as shown by subsequent research or results in this paper. Morris \cite{Morris} is able to characterize new FC-families on six elements and with the help of a computer program exactly characterizes all minimal FC-families on 5 elements. 

Given a family of sets $\mathcal{S}$, we say that $\mathcal{S}$ \emph{generates} (or is a \emph{generator} of) $\mathcal{F}$, denoted by $\gen{\mathcal{S}}:=\mathcal{F}$, if and only if $\mathcal{F}$ is a UC family that contains $\mathcal{S}$, and there exists no UC family $\widetilde{\mathcal{F}} \subset \mathcal{F}$ such that $\mathcal{S} \subseteq \widetilde{\mathcal{F}}$. A generator of a UC family $\mathcal{F}$ is \emph{minimal} if it does not contain a smaller generator of $\mathcal{F}$. Johnson and Vaughan \cite{Unique} show that each UC family has a unique \emph{minimal} generator. In this paper, we are mainly interested in minimal generators of minimal FC-families. Hence from now on, to improve readability, we simply refer to \emph{minimal} generators of FC-families. In order to facilitate the combinatorial analysis of FC-families, Morris \cite{Morris} introduces the following notion. Let $FC(k,n)$ denote the smallest $m$ such that any $m$ of the $k$-sets in $\left\{1,2, \ldots,n\right\}$ generate an FC-family. As proven in Gao and Yu \cite{Gao}, $FC(k,n)$ is always defined for sufficiently large $n$ in relation to $k$. Consequently, Morris \cite{Morris} shows that $FC(3,5)=3$, $FC(4,5)=5$, $FC(3,6)=4$, $7 \leq FC(4,6) \leq 8$, $FC(3,7)\leq 6$ and $FC(4,7) \leq 18$. Such characterizations further facilitate the search for better bounds (or possible counterexamples). Finally, Mari\'c, \v Zivkovi\'c, and Vu\v ckovi\'c \cite{Serbs} formalize a combinatorial search in the interactive theorem prover Isabelle/HOL and show that all families containing four 3-subsets of a 7-set are
FC-families. Although not explicitly mentioned in their paper, their result implies that $FC(3,7)=4$ by the lower bound on the number of 3-sets of Morris \cite{Morris}. In summary, previous research has yielded less than \emph{two dozen} exact characterizations of \emph{minimal} generators of FC-families, with roughly a dozen more characterizations of general FC-families. In light of the above, our main contributions in this paper are the following:
\begin{itemize}
\item We design a general computational framework that is able to precisely characterize FC or Non\textendash FC-families by using exact integer programming and other redundant verification routines, thus providing an algorithmic road-map for settling open questions in Morris \cite{Morris} and Vaughan \cite{Vaughan2}, \cite{Vaughan3}. 
\vspace{0.1cm}
\item In particular we construct an explicit counterexample to a conjecture of Morris~\cite{Morris} about the structure of generators for Non\textendash FC-families. Furthermore we answer in the negative two related questions of
Vaughan~\cite{Vaughan2} and Morris~\cite{Morris} regarding a simplified method for proving the
existence of weights that yield FC-families.
\vspace{0.1cm}
\item In the Appendix we feature over one hundred previously unknown minimal nonisomorphic (under permutations of the ground set) generators of FC-families. We find the first known exact characterizations of minimal generators of FC-families on $8 \leq n \leq 10$.

\iffalse 
\item Given that Algorithm~\ref{row generation} determines exactly whether a given UC family $\mathcal{A}$ if FC or Non\textendash FC for $6 \leq n \leq 10$, lower bounds for previously unknown $FC(k,n)$ in this range become trivial to obtain. Furthermore when coupled with a computer algebra system or graph isomorphism software to obtain the isomorphism types of generators, upper or exact bounds for previously unknown $FC(k,n)$ are obtained in the aforementioned range.
\vspace{0.3cm}
\fi
\end{itemize}
The connection between Frankl’s conjecture and mathematical programming
is well-established in Pulaj, Raymond and Theis~\cite{Pulaj}, where the authors derive
the equivalence of the problem with an integer program and investigate related
conjectures. Furthermore, given an UC family A, Poonen’s Theorem yields a
constructive proof to determine if A is FC or Non–FC in the form of a fractional
polytope with a potentially exponential number of constraints. In general, this
makes it difficult to explicitly state the conditions which determine whether a
given UC family is FC. To overcome this, we design a cutting-plane method
that computes the explicit weights which imply Poonen’s existence conditions.
In particular, this paves the way toward automated discovery of FC-families
by computational integer programming, especially when coupled (as we do in this work) with an exact
rational solver~\cite{Exact} and other verification routines such as the recent work of
Cheung, Gleixner and Steffy~\cite{gleixner2015exact}. Our current implementation\footnote{Final computations are rechecked with CPLEX 12.6.3 \cite{CPLEX}, Gurobi 6.5.2 \cite{Gurobi}, and exact SCIP \cite{Exact}. For $n \geq 8$, we use CPLEX 12.6.3 \cite{CPLEX}, then recheck the results with the rest of the solvers. In addition, the branch and bound tree of exact SCIP \cite{Exact} is verified with VIPR \cite{VIPR}. Our implementation is freely available at \url{https://github.com/JoniPulaj/cutting-planes-UC-families}} in SCIP 3.2.1~\cite{SCIP}
allows us to characterize any FC-family up to 10 elements tested so far.

\section{Poonen's Theorem}
In this paper we are only interested in finite families of finite sets, which we will simply refer to as families of sets. First we will need the following definitions. For two families of sets $ \mathcal{A}$ and $ \mathcal{B}$, let $ \mathcal{A} \uplus \mathcal{B} := \left\{A \cup B \ | \ A \in \mathcal{A}, B \in \mathcal{B} \right\}$. Let $[n]:=\left\{1,2,\ldots,n\right\}$ and let $\mathcal{P}([n])$ denote the power set of $[n]$. Let $\mathcal{F}$ be a family of sets and denote by $U(\mathcal{F})$ the union of all sets in $\mathcal{F}$. For $i \in U(\mathcal{F})$ define $\mathcal{F}_i := \left\{ F \in \mathcal{F} \ | \ i \in F \right\}$.
Poonen's theorem~\cite{Poonen} is central to all approaches for classifying FC-families. In the following to simplify notation  we assume w.l.o.g. that $U(\mathcal{A})=[n]$. 
\begin{thrm}[Poonen 1992] \label{Poonen}
Let $\mathcal{A}$ be a UC family such that $\emptyset \in \mathcal{A}$. The following statements are equivalent:
\vspace{3mm}
 \begin{enumerate}
   \item For every UC family $ \mathcal{F} \supseteq \mathcal{A}$, there exists $i \in [n]$ such that $  |\mathcal{F}_i| \geq |\mathcal{F}|/2$.
\\
 \item There exist nonnegative real numbers $c_1, \ldots , c_n$ with $\sum_{i \in [n]}c_i =1$ such that  for  every UC family $ \mathcal{B} \subseteq \mathcal{P}([n])$ with $\mathcal{B} \uplus  \mathcal{A} = \mathcal{B}$, the following inequality holds

\begin{equation}
\sum_{i\in [n]}c_i|\mathcal{B}_i| \geq |\mathcal{B}|/2 .  \label{poon}
\end{equation}
  \end{enumerate}
\end{thrm}

It is important to note that Poonen's Theorem still holds if $\emptyset \not \in \mathcal{A}$. In this case the condition $\mathcal{B} \uplus \mathcal{A} = \mathcal{B}$ becomes $\mathcal{B} \uplus \mathcal{A} \subseteq \mathcal{B}$. This is an equivalent condition we find in Vaughan \cite{Vaughan1}, \cite{Vaughan2}, \cite{Vaughan3}. %In the following we consider only nonempty matrices, i.e., matrices which have at least one row and one column.
For a fixed UC family $\mathcal{A}$ such that $\emptyset \in \mathcal{A}$, the second statement in Theorem~\ref{Poonen} can be seen as a polyhedron defined as the following:
\[  P^{\mathcal{A}}:=\left\{ y \in \mathbb{R}^n \ \vline \ \begin{array}{ll}
         \sum_{i\in [n]}y_i = 1;\\
         \\
         \sum_{i\in [n]}y_i|\mathcal{B}_i| \geq |\mathcal{B}|/2 & \mbox{ $\forall \text{ UC } \mathcal{B} \subseteq \mathcal{P}([n]) : \mathcal{B} \uplus \mathcal{A} = \mathcal{B}$};\\
         \\
          y_i \geq 0 & \ \mbox{$\forall i \in [n]$};\end{array} \right\} \] 
Furthermore since the coefficients (and the right-hand side vector) are all rational, if $P^{\mathcal{A}}$ is nonempty, we can safely assume (via Fourier-Motzkin elimination) that it contains a rational vector. This is a very well-known result (for more details see the excellent exposition of Aigner and Ziegler \cite[pp.66]{Aigner}) which we formally state as follows for completeness and reference.
%\begin{definition}
%A polyhedron $P$  such that $ P = \left\{x \in \mathbb{R}^n  \ | \ Ax\leq b\right\}$ is called a rational polyhedron, if and only if  $A \in \mathbb{Q}^{m\times n}$ and $b \in \mathbb{Q}^m$.
%\end{definition}
\begin{prop} \label{FM}
Let $P$ be a nonempty rational polyhedron. Then $P$ contains a rational vector.
\end{prop}
 We can use the simplex or interior point methods to find a feasible point of $P^{\mathcal{A}}$, or show that one does not exist via Farkas' Lemma. Suppose $P^{\mathcal{A}}$ is nonempty. Then we can scale any rational vector contained in $P^{\mathcal{A}}$ and arrive at an integer vector. In particular, for reasons that we outline in Section~\ref{sec:Relaxation}, we want to choose a rational vector such that the $\ell_1$ norm of the resulting integer vector is as small as possible. This explains the objective function of the following integer program.  Let $I^{\mathcal{A}}$ denote the following integer program:
\begin{align*}\
\min &\sum_{i \in [n]} z_i &\\
\textrm{s.t. }&\sum_{i\in [n]}z_i|\mathcal{B}_i| \geq \left(|\mathcal{B}|/2\right) \sum_{i \in [n]} z_i  & \forall \mathcal{B} \subseteq \mathcal{P}([n]) : \mathcal{B} \uplus \mathcal{A} = \mathcal{B} \\
&\sum_{i \in [n]} z_i \geq 1 \\
& z_i  \in \mathbb{Z}_{\geq 0} & \forall i \in [n]
\end{align*}\\
A \emph{feasible} solution of $I^{\mathcal{A}}$ is a vector $\bar{z} \in \mathbb{Z}^n_{\geq 0}$ such that $\bar{z}$ satisfies all the given inequalities of $I^{\mathcal{A}}$.
\begin{prop}\label{equivalence}
Let $\mathcal{A}$ be a UC family such that $\emptyset \in \mathcal{A}$. Then $P^{\mathcal{A}}$ is nonempty if and only if there exists a feasible solution of $I^{\mathcal{A}}$.
\end{prop}

\begin{proof}
Suppose $P^{\mathcal{A}}$ is nonempty and let $\bar{y} \in P^{\mathcal{A}}$. From Proposition~\ref{FM} we can safely assume that $\bar{y} \in \mathbb{Q}^n_{\geq 0}$, i.e., $\bar{y} = (\bar{y}_1=\frac{a_1}{b_1}, \bar{y}_2 = \frac{a_2}{b_2} , \ldots, \bar{y}_n  = \frac{a_n}{b_n})$ such that $b_i \geq 1$ for all $i \in [n]$. Let $g \in \mathbb{Z}_{\geq 0}$ such that $g = lcm(b_1,b_2,\ldots, b_n)$, and let $\bar{z}_i \in \mathbb{Z}_{\geq 0}$ such that $\bar{z}_i = g\bar{y}_i$ for all $i \in [n]$. Define $\bar{z}:=(\bar{z}_1,\bar{z}_2, \ldots, \bar{z}_n)$. It follows that $\bar{z} \in \mathbb{Z}^n_{\geq 0}$ is a feasible solution of $I^{\mathcal{A}}$. \\ For the other direction, suppose the vector $\bar{z} \in \mathbb{Z}^n_{\geq 0}$ is a feasible solution of $I^{\mathcal{A}}$. Let $\bar{z} = (\bar{z}_1, \bar{z}_2, \ldots, \bar{z}_n)$. Define $\bar{y}_i := \bar{z}_i /(\sum_{i \in [n]} \bar{z}_i)$ for all $i \in [n]$ and $\bar{y}:=(\bar{y}_1,\bar{y}_2, \ldots, \bar{y}_n)$. It follows that $\bar{y}\in P^{\mathcal{A}}$. 
\end{proof}
We need the following corollary of Poonen's Theorem, a version of which is already noted in Morris \cite{Morris}. We formalize it again here for clarity and reference.
\begin{cor}\label{weight}
Let $\mathcal{A}$ be a UC family such that $\emptyset \in \mathcal{A}$. The following statements are equivalent:
\begin{enumerate}

\item For every UC family $ \mathcal{F} \supseteq \mathcal{A}$, there exists $i \in [n]$ such that $  |\mathcal{F}_i| \geq |\mathcal{F}|/2$.

\item There exist $c_i \in \mathbb{Q}_{\geq 0}$ for all $i\in [n]$ with $\sum_{i\in [n]}c_i = 1$, such that for every UC family $ \mathcal{B} \subseteq \mathcal{P}([n])$ with $\mathcal{B} \uplus \mathcal{A} = \mathcal{B}$, $\sum_{S\in \mathcal{B}}\left(\sum_{i \in S}c_i - \sum_{i \notin S}c_i \right) \geq 0$ holds.

\end{enumerate}
\end{cor}

\begin{proof}
Fix a UC family $ \mathcal{B} \subseteq \mathcal{P}([n])$ with $\mathcal{B} \uplus \mathcal{A} = \mathcal{B}$. Then the following holds,
\begin{flalign*} 
\sum_{S\in \mathcal{B}}\left(\sum_{i \in S}c_i - \sum_{i \notin S}c_i \right) &= 2\sum_{S\in \mathcal{B}}\sum_{i \in S}c_i  -\sum_{S\in \mathcal{B}}\left(\sum_{i \notin S}c_i + \sum_{i \in S}c_i \right)\\
                                                                                                                       &= 2\sum_{S\in \mathcal{B}}\sum_{i \in S}c_i  -\sum_{S\in \mathcal{B}}\sum_{i \in [n]}c_i\\
                                                                                                                       &= 2\sum_{i\in [n]}c_i|\mathcal{B}_i| - |\mathcal{B}|\sum_{i \in [n]}c_i \geq 0\\ 
                                                                                                                     &\Longleftrightarrow  \sum_{i\in [n]}c_i|\mathcal{B}_i|\geq  |\mathcal{B}|/2. \\
\end{flalign*}\\
Since the above holds for every  UC family $ \mathcal{B} \subseteq \mathcal{P}([n])$ with $\mathcal{B} \uplus \mathcal{A} = \mathcal{B}$, the desired result follows from Poonen's Theorem. 
\end{proof}
Proposition~\ref{equivalence} shows that if $P^{\mathcal{A}}$ is nonempty we can simply scale a rational vector contained in it and arrive at an integer vector. Then the proof of the previous corollary implies the following.
\begin{cor}\label{integer}
Let $\mathcal{A}$ be a UC family such that $\emptyset \in \mathcal{A}$. The following statements are equivalent:
\begin{enumerate}

\item For every UC family $ \mathcal{F} \supseteq \mathcal{A}$, there exists $i \in [n]$ such that $  |\mathcal{F}_i| \geq |\mathcal{F}|/2$.

\item There exist $c_i \in \mathbb{Z}_{\geq 0}$ for all $i\in [n]$ with $\sum_{i\in [n]}c_i \geq 1$, such that for every UC family $ \mathcal{B} \subseteq \mathcal{P}([n])$ with $\mathcal{B} \uplus \mathcal{A} = \mathcal{B}$, $\sum_{S\in \mathcal{B}}\left(\sum_{i \in S}c_i - \sum_{i \notin S}c_i \right) \geq 0$ holds.

\end{enumerate}
\end{cor}
\begin{proof}
It is sufficient to follow the proof of Corollary~\ref{weight} with $c_i \in \mathbb{Z}_{\geq 0}$ for all $i\in [n]$ such that $\sum_{i\in [n]}c_i \geq 1$. Then we arrive at the following
\begin{equation}
2\sum_{i\in [n]}c_i|\mathcal{B}_i| - |\mathcal{B}|\sum_{i \in [n]}c_i \geq 0 \nonumber.
\end{equation}
The desired result is implied from Proposition~\ref{equivalence} and Poonen's Theorem. 
\end{proof}
%Depending on $\mathcal{A}$ and $n$, there may be a large number of inequalites based on $ \mathcal{B} \subseteq \mathcal{P}([n])$ such that $\mathcal{B} \uplus \mathcal{A} = \mathcal{B}$. In general, this makes the explicit enumeration of such inequalities impractical. Fortunately we have at our disposal a well-established technique that handles precisely this type of problem, namely the cutting-plane method.%
%A \emph{separation oracle} for a polyhedron $P \subset \mathbb{R}^n$ is an algorithm that, queried on $x \in \mathbb{R}^n$, either asserts that $x \in P$ or returns $h \in  \mathbb{R}^n$ such that $hy < hx$ for all $y \in P$. 
%As a result, the following question is of central interest: Given a feasible point $p^* \in \tilde{P}_c \supseteq P^{\mathcal{A}}$, where $\tilde{P}_c$ is the polyhedron defined by a small subset of all possible inequalites of the type $ \mathcal{B} \subseteq \mathcal{P}([n])$ with $\mathcal{B} \uplus \mathcal{A} = \mathcal{B}$ and nonnegative $c_i$ that sum to one, how d we exhibit a UC family (coefficients for an inequality of type \ref{poon}) which separates $p^*$ from $P^{\mathcal{A}}$ or show that no such family exists? It is clear that the same applies to $I^{\mathcal{A}}$.%
From now on, we can base relevant arguments (when convenient) on real, rational or integer vectors.
\begin{cor}\label{rqz}
Let $\mathcal{A}$ be a UC family such that $\emptyset \in \mathcal{A}$. The following statements are equivalent:
\begin{enumerate}
\item  $\mathcal{A}$ is an FC-family.
\item There exist $c_i \in \mathbb{R}_{\geq 0}$ for all $i\in [n]$ with $\sum_{i\in [n]}c_i = 1$, such that for every UC family $ \mathcal{B} \subseteq \mathcal{P}([n])$ with $\mathcal{B} \uplus \mathcal{A} = \mathcal{B}$, $\sum_{i\in [n]}c_i|\mathcal{B}_i| \geq |\mathcal{B}|/2$ holds.
\item There exist $c_i \in \mathbb{Q}_{\geq 0}$ for all $i\in [n]$ with $\sum_{i\in [n]}c_i = 1$, such that for every UC family $ \mathcal{B} \subseteq \mathcal{P}([n])$ with $\mathcal{B} \uplus \mathcal{A} = \mathcal{B}$, $\sum_{S\in \mathcal{B}}\left(\sum_{i \in S}c_i - \sum_{i \notin S}c_i \right) \geq 0$ holds.
\item There exist $c_i \in \mathbb{Z}_{\geq 0}$ for all $i\in [n]$ with $\sum_{i\in [n]}c_i \geq 1$, such that for every UC family $ \mathcal{B} \subseteq \mathcal{P}([n])$ with $\mathcal{B} \uplus \mathcal{A} = \mathcal{B}$, $\sum_{S\in \mathcal{B}}\left(\sum_{i \in S}c_i - \sum_{i \notin S}c_i \right) \geq 0$ holds.
\item There exist $c_i \in \mathbb{Z}_{\geq 0}$ for all $i\in [n]$ with $\sum_{i\in [n]}c_i \geq 1$, such that for every UC family $ \mathcal{B} \subseteq \mathcal{P}([n])$ with $\mathcal{B} \uplus \mathcal{A} = \mathcal{B}$, $\sum_{i\in [n]}c_i|\mathcal{B}_i| \geq \left(|\mathcal{B}|/2\right) \sum_{i \in [n]} c_i$ holds.
\end{enumerate}
\end{cor} 
\begin{proof}
$(1) \iff (2)$ from Poonen's Theorem. $(1) \iff (3)$ from Corollary~\ref{weight}. $(1) \iff (4)$ from Corollary~\ref{integer}. $(2) \iff (5)$ from Proposition~\ref{equivalence}. 
\end{proof}
In the next proposition, we show that for FC or Non\textendash FC-families we can always assume (when convenient) that the empty set is present. 
\begin{prop}\label{emptynot}
Let $\mathcal{A}$ be a UC family such that $\emptyset \in \mathcal{A}$. Then $\mathcal{A}$ is an FC-family if and only if $\mathcal{A} \setminus \left\{\emptyset\right\}$ is an FC-family.
\end{prop}
\begin{proof}
Let $\mathcal{A}$ be a UC family such that $\emptyset \in \mathcal{A}$. Define $\widetilde{\mathcal{A}}:= \mathcal{A} \setminus \left\{\emptyset\right\}$. 
\\Suppose $\mathcal{A}$ is an FC-family. Then for each UC family $\mathcal{F} \supseteq \mathcal{A}$ there exists $i \in U(\mathcal{A})$ such that $|\mathcal{F}_i| \geq |\mathcal{F}|/2$. Hence $\mathcal{F}\setminus \left\{\emptyset\right\}$ also satisfies Frankl's conjecture. It follows that $\widetilde{\mathcal{A}}$ is an FC-family. \\For the other direction, suppose $\widetilde{\mathcal{A}}$ is an FC-family and let $\mathcal{F}$ be a UC family such that $\mathcal{F}\supseteq \mathcal{A}$. Then $\mathcal{F}\supseteq \widetilde{\mathcal{A}}$. Therefore there exists $i \in U(\widetilde{\mathcal{A}})$ such that $|\mathcal{F}_i| \geq |\mathcal{F}|/2$. Since $U(\widetilde{\mathcal{A}})= U(\mathcal{A})$, it follows that $\mathcal{A}$ is an FC-family. 
\end{proof}
\subsection{A Cutting-Plane Method for Poonen's Theorem}
As mentioned in the introduction, the main obstacle in using Poonen's Theorem to characterize FC-families is the potentially exponential number of constraints in $P^{\mathcal{A}}$ or (equivalently) $I^{\mathcal{A}}$. Therefore, our main goal in the rest of this section is to precisely define a method for starting with a small subset of the constraints that define $P^{\mathcal{A}}$ or $I^{\mathcal{A}}$ and then generate more constraints as needed. First we define a set of integer vectors contained in a polyhedron that determines, when the set is empty, that a given rational vector satisfies the second condition of Poonen's Theorem (this is Proposition~\ref{FranklIntProp}). Then we show that the set above is nonempty if and only if a given rational vector does not satisfy the second condition of Poonen's Theorem (this is Theorem~\ref{oracle}). Finally, this gives rise to an algorithm that determines whether a given $\mathcal{A}$ is FC or Non\textendash FC.

Corollary \ref{rqz} combined with the integer programming approach to UC families in in Pulaj, Raymond and Theis~\cite{Pulaj}, provides the background of our method.
Fix a UC family $\mathcal{A}$ such that $\emptyset \in \mathcal{A}$. As previously, we may assume $U(\mathcal{A})=[n]$. Let $c \in \mathbb{Z}^n_{\geq 0}$ such that $\sum_{i\in [n]}c_i \geq 1$. With every set $S \in \mathcal{P}([n])$, we associate a variable $x_S$, i.e, a component of a vector $x \in \mathbb{R}^{2^n}$ indexed by $S$. Given a family of sets $\mathcal{F} \subseteq \mathcal{P}([n])$, let $\mathcal{X}^{\mathcal{F}} \in \mathbb{R}^{2^n}$ denote the incidence vector of $\mathcal{F}$ defined (component-wise) as 
\[ \mathcal{X}^\mathcal{F}_S := \left\{ \begin{array}{ll}
         1 & \mbox{if $S \in \mathcal{F}$},\\
         0 & \mbox{if $S \not \in \mathcal{F}$}.\end{array} \right. \] Hence every family of sets $\mathcal{F} \subseteq \mathcal{P}([n])$ corresponds to a unique zero-one vector in $\mathbb{R}^{2^n}$ and vice versa. Let
 $X(\mathcal{A},c)$ denote the set of integer vectors contained in the polyhedron defined by the following inequalities:
\begin{align}\
& x_S+x_T \leq 1 + x_{S \cup T} & \tiny{\forall S \in \mathcal{P}([n]), \forall T \in \mathcal{P}([n])} \label{one}\\
&\sum_{S\in \mathcal{P}([n])}\left(\sum_{i \in S}c_i - \sum_{i \notin S}c_i \right)x_S +1 \leq 0 \label{two}\\
& x_S \leq x_{A\cup S} &  \forall S \in  \mathcal{P}([n]), \forall A \in \mathcal{A} \label{three}\\
& 0\leq x_S \leq 1 & \forall S\in \mathcal{P}([n]) \label{four}
\end{align}\\
Suppose $X(\mathcal{A},c)$ is nonempty and let $\bar{x} \in X(\mathcal{A},c)$. Then $\bar{x}= \mathcal{X}^{\mathcal{B}}$ for some family of sets $\mathcal{B}$ such that $\mathcal{B} \subseteq \mathcal{P}([n])$. Inequalities~\eqref{one} ensure that the chosen family $\mathcal{B}$ is UC, and we denote them as  UC inequalities. Inequalities~\eqref{three} ensure that $\mathcal{B} \, \uplus \, \mathcal{A} = \mathcal{B}$, and we denote them as Fixed-Set (FS) inequalities. We denote Inequality~\eqref{two} as the Weight Vector (WV) inequality and we explain it in the next proposition.
\begin{prop} \label{FranklIntProp}
Let $\mathcal{A}$ be a UC family such that $\emptyset \in \mathcal{A}$, and let $c \in \mathbb{Z}^n_{\geq 0}$ such that $\sum_{i\in [n]}c_i \geq 1$. If $X(\mathcal{A},c) = \emptyset$, then $\mathcal{A}$ is an FC-family.

\end{prop}
\begin{proof} Suppose that $X(\mathcal{A},c) = \emptyset$. Let  $Y(\mathcal{A},c)$ be defined as the set of integer vectors contained in the polyhedron defined by Inequalities~\eqref{one},~\eqref{three} and ~\eqref{four}. For any UC family $\mathcal{B} \subseteq \mathcal{P}([n])$ such that $\mathcal{B} \uplus \mathcal{A} = \mathcal{B}$, we arrive at $\mathcal{X}^{\mathcal{B}} \in Y(\mathcal{A},c)$.  
Therefore if $X(\mathcal{A},c) = \emptyset$ this implies there exists no UC family $\mathcal{B} \subseteq \mathcal{P}([n])$ with $\mathcal{B} \uplus \mathcal{A} = \mathcal{B}$  such that:
\begin{equation}
\sum_{S\in \mathcal{B}}\left(\sum_{i \in S}c_i - \sum_{i \notin S}c_i \right) \leq -1. \nonumber
\end{equation}
Since $c_i \in \mathbb{Z}_{\geq 0}$ for all $i \in [n]$, this implies that for each UC family $\mathcal{B} \subseteq \mathcal{P}([n])$ with $\mathcal{B} \uplus \mathcal{A} = \mathcal{B}$, the following inequality holds:
\begin{equation}
\sum_{S\in \mathcal{B}}\left(\sum_{i \in S}c_i - \sum_{i \notin S}c_i \right) \geq 0. \nonumber
\end{equation}
Corollary~\ref{rqz} implies that each UC family $\mathcal{F}$ such that $\mathcal{F} \supseteq \mathcal{A}$, satisfies Frankl's conjecture.
\end{proof}
\iffalse
\begin{corollary}
Let $\mathcal{A}$ be a UC family such that $\emptyset \in \mathcal{A}$, and let $c \in \mathbb{Z}^n_{\geq 0}$ such that $\sum_{i\in [n]}c_i \geq 1$. Then $X(\mathcal{A},c) = \emptyset$ if and only if  $FC(\mathcal{A} \setminus \emptyset,\mathcal{C})_{n} = \emptyset$.
\end{corollary}
\begin{proof}
Follows immediately from proposition \ref{emptynot} and proposition \ref{FranklIntProp}.
\end{proof}
\fi
A natural candidate for checking whether $X(\mathcal{A},c)$ is empty (or not), for some $\mathcal{A}$ and $c$, is a standard branch and bound algorithm. Hence we define an appropriate integer program related to $X(\mathcal{A},c)$ and solve it in a general purpose integer programming solver as specified in the introduction. However in order to prove that a ``candidate" UC family is an FC-family, we need a vector $c$  which yields an empty $X(\mathcal{A},c)$, if such a vector exists. Thus we turn our attention to the relation between $X(\mathcal{A},c)$ and $P^{\mathcal{A}}$, for a given $\mathcal{A}$ and $c$. First we need the following basic definition.
\begin{definition}
A valid inequality $\pi^{T}x \geq \pi_{0}$ for a set $X \subseteq \mathbb{R}^n$ is violated by a vector $\bar{x} \in \mathbb{R}^n$ if and only if $\pi^{T}\bar{x} < \pi_{0}$.
\end{definition}
Given $c \in \mathbb{Z}^n_{\geq 0}$ such that $\sum_{i\in [n]}c_i \geq 1$, we define $\bar{y}$ as $c$ normalized by its $\ell_1$ norm. Thus $\bar{y}= c/\sum_{i\in [n]}c_i$. By definition we arrive at $\bar{y} \in \mathbb{Q}^n_{\geq 0}$ such that $\sum_{i\in [n]}\bar{y}_i = 1$.
\begin{thrm} \label{oracle}
Let $\mathcal{A}$ be a UC family such that $\emptyset \in \mathcal{A}$ and let $c \in \mathbb{Z}^n_{\geq 0}$ such that $\sum_{i\in [n]}c_i \geq 1$. Then $X(\mathcal{A},c)$ is nonempty if and only if there exists a valid inequality of $P^{\mathcal{A}}$ that is violated by $\bar{y}$.
\end{thrm}
\begin{proof}

Suppose $X(\mathcal{A},c)$ is nonempty. Hence there exists  $\bar{x} \in X(\mathcal{A},c)$ such that $\bar{x}= \mathcal{X}^{\mathcal{B}}$ for some $\mathcal{B} \subseteq \mathcal{P}([n])$. $\mathcal{B}$ is a UC family since the corresponding UC inequalities are satisfied. Furthermore, for each $B \in \mathcal{B}$ and for each $A \in \mathcal{A}$, it follows that $A \cup B \in \mathcal{B}$ since all the corresponding FS inequalities are satisfied. Hence we see that $\mathcal{B} \uplus \mathcal{A}= \mathcal{B}$. Therefore $\mathcal{B}$ yields the coefficients (and the right-hand side scalar) of the following valid inequality for $P^{\mathcal{A}}$, 
\begin{equation}
 \sum_{i\in [n]}y_i|\mathcal{B}_i|\geq  |\mathcal{B}|/2 \nonumber.
\end{equation} 
Since $\mathcal{X}^{\mathcal{B}}\in X(\mathcal{A},c)$ implies the WV inequality is satisfied, we arrive at the following,
\begin{equation}
\sum_{S\in \mathcal{B}}\left(\sum_{i \in S}c_i - \sum_{i \notin S}c_i \right) \leq -1 \nonumber.
\end{equation}
Combining the above with the proof of Corollary~\ref{integer} we arrive at the following inequality,
\begin{equation}
2\sum_{i\in [n]}c_i|\mathcal{B}_i| - |\mathcal{B}|\sum_{i \in [n]}c_i \leq -1 \nonumber.
\end{equation}
Adding $|\mathcal{B}|\sum_{i \in [n]}c_i$ to both sides of the above and dividing by $2\sum_{i\in [n]}c_i$, we arrive at 
\begin{equation}
\sum_{i \in [n]}\frac{c_i}{\sum_{i \in [n]}c_i}|\mathcal{B}_i| \leq \frac{-1}{2\sum_{i \in [n]}c_i} + \frac{|\mathcal{B}|}{2}\nonumber
\end{equation}
and because $\frac{-1}{2\sum_{i \in [n]}c_i} < 0$, and $\frac{c_i}{\sum_{i \in [n]}c_i}=\bar{y}_i$ for each $i \in [n]$, it follows that
\begin{equation}
\sum_{i \in [n]}\bar{y}_i|\mathcal{B}_i| < |\mathcal{B}|/2\nonumber.
\end{equation}
For the other direction, suppose $X(\mathcal{A},c) = \emptyset$. Following the proof of Proposition~\ref{FranklIntProp}  we see that for each $\mathcal{B} \subseteq \mathcal{P}([n])$ such that $\mathcal{B} \uplus \mathcal{A} = \mathcal{B}$, the following inequality holds:
\begin{equation}
\sum_{S\in \mathcal{B}}\left(\sum_{i \in S}c_i - \sum_{i \notin S}c_i \right) \geq 0. \nonumber
\end{equation}
Hence, Corollary~\ref{rqz} implies that $\bar{y} \in P^{\mathcal{A}}$.
\end{proof}
We determined that a nonempty $X(\mathcal{A},c)$ implies a violated inequality for $P^{\mathcal{A}}$. However for a given $\mathcal{A}$ and $c$, there may be many such violated inequalities. This leads to the notion of a maximally violated inequality, which we define below. This notion is based on the intuition that a maximally violated inequality is ``farthest" away from $P^{\mathcal{A}}$, and hence adding it to a subset of the constraints of $P^{\mathcal{A}}$ should get us ``closest" to $P^{\mathcal{A}}$.\footnote{Indeed, from a computational perspective, for all the tested UC families in this paper, using this notion for the objective function of $IP(\mathcal{A},c)$ leads to the fewest number of iterations of Algorithm~\ref{row generation}, where we use $IP(\mathcal{A},c)$ instead of $X(\mathcal{A},c)$.}
\begin{definition}
Let $\mathcal{A},\mathcal{B} \subseteq \mathcal{P}([n])$, be UC families such $\mathcal{B} \uplus \mathcal{A} = \mathcal{B}$ and $\emptyset \in \mathcal{A}$. A valid inequality $\sum_{i\in [n]}{y}_i|\mathcal{B}_i|\geq  |\mathcal{B}|/2$ for $P^{\mathcal{A}}$ is maximally violated by a vector $\bar{y} \in \mathbb{Q}^n_{\geq 0}$ such that $\sum_{i\in [n]}\bar{y}_i = 1$ if and only if  for each violated valid inequality  $\sum_{i\in [n]}y_i|\mathcal{D}_i|\geq  |\mathcal{D}|/2$ such that $\mathcal{D} \subseteq \mathcal{P}([n])$ is a UC family and $\mathcal{D} \uplus \mathcal{A} = \mathcal{D}$, the following inequalities $(|\mathcal{B}|/2 - \sum_{i\in [n]}\bar{y}_i|\mathcal{B}_i| ) \geq (|\mathcal{D}|/2 - \sum_{i\in [n]}\bar{y}_i|\mathcal{D}_i|) > 0$ hold.
\end{definition}
Let $\mathcal{A}$ be a UC family such that $\emptyset \in \mathcal{A}$. Furthermore, let $c \in \mathbb{Z}^n_{\geq 0}$ such that $\sum_{i\in [n]}c_i \geq 1$. Denote by $IP(\mathcal{A},c)$ the following integer program:
\begin{align*}\
\max & \sum_{i\in [n]}c_i \left(\sum_{S \in \mathcal{P}([n])}x_S - 2\sum_{S \in \mathcal{P}([n]): i \in S}x_S\right) &\\
\textrm{s.t. } & x \in X(\mathcal{A},c)
\end{align*}\\
An integer vector $\bar{x} \in \mathbb{R}^{2^n}$ is a \emph{feasible} solution of $IP(\mathcal{A},c)$ if and only if $\bar{x} = \mathcal{X}^{\mathcal{B}}$ for some UC family $\mathcal{B} \subseteq \mathcal{P}([n])$ such that $\mathcal{B} \uplus \mathcal{A} = \mathcal{B}$ and $\mathcal{X}^{\mathcal{B}}$ satisfies the WV inequality. $IP(\mathcal{A},c)$ is \emph{infeasible} if and only if there exists no feasible solution of $IP(\mathcal{A},c)$.
$\mathcal{X}^{\mathcal{B}}$ is an \emph{optimal} solution of $IP(\mathcal{A},c)$ if and only if $\mathcal{X}^{\mathcal{B}}$ is a feasible solution of $IP(\mathcal{A},c)$, and for any other feasible solution $\mathcal{X}^{\mathcal{D}}$ of $IP(\mathcal{A},c)$, we arrive at
\begin{equation}
\sum_{S\in \mathcal{B}}\sum_{i \in [n]}c_i - 2\sum_{S\in \mathcal{B}}\sum_{i \in S}c_i \geq \sum_{S\in \mathcal{D}}\sum_{i \in [n]}c_i - 2\sum_{S\in \mathcal{D}}\sum_{i \in S}c_i \nonumber.
\end{equation}
\begin{thrm}
Let $\mathcal{A}$ be a UC family such that $\emptyset \in \mathcal{A}$, and let $c \in \mathbb{Z}^n_{\geq 0}$ such that $\sum_{i\in [n]}c_i \geq 1$. Suppose $\mathcal{X}^{\mathcal{B}}$ is an optimal solution of $IP(\mathcal{A},c)$. Then the valid inequality $\sum_{i\in [n]}y_i|\mathcal{B}_i|\geq  |\mathcal{B}|/2$ for $P^{\mathcal{A}}$ is maximally violated by $\bar{y}$.
\end{thrm} 
\begin{proof}
Suppose $\mathcal{X}^{\mathcal{B}}$ is an optimal solution of $IP(\mathcal{A},c)$. Then the following inequality holds:
\begin{equation}
\sum_{S\in \mathcal{B}}\left(\sum_{i \in S}c_i - \sum_{i \notin S}c_i \right) \leq -1 \nonumber.
\end{equation}
Following the proof of Corollary~\ref{integer} we arrive the following:
\begin{equation}
\sum_{S\in \mathcal{B}}\sum_{i \in [n]}c_i - 2\sum_{S\in \mathcal{B}}\sum_{i \in S}c_i \geq 1 \nonumber.
\end{equation}
Suppose $\mathcal{X}^{\mathcal{D}}$ is a feasible solution of $IP(\mathcal{A},c)$. Then the following holds:
\begin{equation}
\sum_{S\in \mathcal{B}}\sum_{i \in [n]}c_i - 2\sum_{S\in \mathcal{B}}\sum_{i \in S}c_i \geq \sum_{S\in \mathcal{D}}\sum_{i \in [n]}c_i - 2\sum_{S\in \mathcal{D}}\sum_{i \in S}c_i \geq 1\nonumber.
\end{equation}
Rewriting the inequalities above as in the proof of Corollary~\ref{integer} combined with the proof of Theorem~\ref{oracle} we arrive at
\begin{equation}
\frac{|\mathcal{B}|}{2}-\sum_{i \in [n]}\frac{c_i}{\sum_{i \in [n]}c_i}|\mathcal{B}_i|\geq \frac{|\mathcal{D}|}{2}-\sum_{i \in [n]} \frac{c_i}{\sum_{i \in [n]}c_i}|\mathcal{D}_i|  \geq \frac{1}{2\sum_{i \in [n]}c_i} \nonumber.
\end{equation}
Finally, this implies that the following holds:
\begin{equation}
 (|\mathcal{B}|/2 - \sum_{i\in [n]}\bar{y}_i|\mathcal{B}_i| ) \geq (|\mathcal{D}|/2 - \sum_{i\in [n]}\bar{y}_i|\mathcal{D}_i|) > 0 \nonumber.
\end{equation}

\end{proof}
Given a UC family $\mathcal{A}$, the following algorithm finds a rational vector that satisfies the second condition of Poonen's Theorem, or an infeasible subset of the constraints that define $P^{\mathcal{A}}$. The former proves that $\mathcal{A}$ is FC, whereas the latter proves that $\mathcal{A}$ is Non\textendash FC. Using Proposition~\ref{equivalence} with appropriate adjustments in the algorithm below we may search for an infeasible subset of the constraints that define $I^{\mathcal{A}}$ instead of $P^{\mathcal{A}}$. Furthermore, we may use $IP(\mathcal{A},c)$ instead of $X(\mathcal{A},c)$. For a given vector $\bar{y} \in \mathbb{Q}^n_{\geq 0}$ such that $\bar{y}=(\frac{a_1}{b_1}, \frac{a_2}{b_2}, \ldots, \frac{a_n}{b_n})$, we safely assume that $b_i \geq 1$ for all $i \in [n]$.

\vspace{5mm}

\begin{algorithm}[H]\label{row generation}
 \SetKwInOut{Input}{Input}\SetKwInOut{Output}{Output}
 \Input{A UC family $\mathcal{A}$ such that $U(\mathcal{A})=[n]$ and $\emptyset \in \mathcal{A}$}
 \Output{$\mathcal{A}$ is an FC-family, or $\mathcal{A}$ is a Non\textendash FC-family}
 $H \leftarrow \left( \sum_{i\in [n]}y_i = 1, \ y_i \geq 0 \ \forall i \in [n] \right)$ \DontPrintSemicolon

 \While{$\exists \ \bar{y} \in H$ such that $\bar{y}=(\frac{a_1}{b_1}, \frac{a_2}{b_2}, \ldots, \frac{a_n}{b_n}) \in \mathbb{Q}^n_{\geq 0}$}{
  
   $g \leftarrow lcm(b_1,b_2, \ldots,b_n)$ \; \DontPrintSemicolon
   $c \leftarrow g\bar{y}$ \; \DontPrintSemicolon
  \If {$\exists \ \mathcal{X}^{\mathcal{B}} \in X(\mathcal{A},c)$}{
  $H \leftarrow H \cap \left(\sum_{i\in [n]}y_i|\mathcal{B}_i| \geq |\mathcal{B}|/2 \right)$\

  }
  \Else{
  \Return{} $\ \mathcal{A}$ is an FC-family
  }
}

\Return{}$ \ \mathcal{A}$ is a Non\textendash FC-family\
 \caption{Cutting planes for FC-families}
\end{algorithm}
\vspace{3mm}
\begin{thrm}
Let $\mathcal{A}$ be a UC family such that $U(\mathcal{A})=[n]$ and $\emptyset \in \mathcal{A}$. Then Algorithm \ref{row generation} correctly determines if $\mathcal{A}$ is an FC-family or Non\textendash FC-family.
\end{thrm}
\begin{proof}
It is clear Algorithm~\ref{row generation} finitely terminates. Furthermore, if $H$ is nonempty, then by Proposition~\ref{FM} it contains a rational vector.
Let $\mathcal{A}$ be a UC family such that $U(\mathcal{A})=[n]$ and $\emptyset \in \mathcal{A}$. Suppose $\mathcal{A}$ is an FC-family. By the definition of an FC-family and by Poonen's Theorem there exist $c_i \geq 0$ for all $i \in [n]$, such that $\sum_{i \in [n]}c_i = 1$, which satisfy all Inequalities~\eqref{poon}. Therefore $P^{\mathcal{A}}$ is nonempty and consequently $H$ is nonempty. This implies that at some iteration of Algorithm~\ref{row generation}, by Theorem~\ref{oracle} we arrive at $\bar{y} \in P^{\mathcal{A}}$, otherwise Algorithm~\ref{row generation} determines an infeasible system of constraints that defines $H$ and we arrive at a contradiction.
Suppose $\mathcal{A}$ is a Non\textendash FC-family. By the definition of a Non\textendash FC-family and Poonen's Theorem, this implies there exist no $c_i \geq 0$ for all $i \in [n]$ with $\sum_{i \in [n]}c_i = 1$ that satisfy all Inequalities~\eqref{poon}. By Theorem~\ref{oracle} during all the iterations of Algorithm~\ref{row generation} we have that $\bar{y} \not \in P^{\mathcal{A}}$, otherwise we arrive at a contradiction. Therefore Algorithm~\ref{row generation} terminates when it determines a system of constraints that define $H$ such that $H = \emptyset$, which implies that $P^{\mathcal{A}} = \emptyset$.
\end{proof}
Algorithm~\ref{row generation} becomes our main tool for determining whether certain UC families are FC or Non\textendash FC. This in turn allows us to answer other questions of interest. In the next section we narrow our focus on valid inequalities for $IP(\mathcal{A},c)$. Our interest in these is mainly \emph{practical}, since solving $IP(\mathcal{A},c)$ in a general purpose integer programming solver is how we determine if $X(\mathcal{A},c)$ is empty or not.
\section{Valid Inequalities for $X(\mathcal{A},c)$}\label{sec:Inequalities}
From the perspective of computational integer programming, valid inequalities may be considered effective if\textemdash among other things\textemdash they lead to a smaller branch and bound tree. For all the results that we feature in this paper, adding a subset of the following inequalities to the root node of a given instance of $IP(\mathcal{A},c)$ significantly reduces the size of the resulting branch and bound tree. This is particularly important in the implementation of Algorithm~\ref{row generation} which features $IP(\mathcal{A},c)$. Since the algorithm may iterate many times, speeding up the solution process of $IP(\mathcal{A},c)$ becomes crucial. Once Algorithm~\ref{row generation} determines whether a given $\mathcal{A}$ is an FC or Non\textendash FC-family, separate rounds of verifications take place in a number of different solvers as mentioned in the introduction. If the given family $\mathcal{A}$ is FC, then automated verifications are carried out in an exact rational solver \cite{Exact} and VIPR \cite{VIPR} which do not make use of the following inequalities, thus allowing for, if necessary, a straightforward check of the input files.\footnote{This is important because it means that the interested reader does not need to rely on the implementation of Algorithm~\ref{row generation}, and in particular the generation of FC-chain inequalities, in order to computationally reproduce the results featured in this paper. To check that a given $\mathcal{A}$ is FC, a reader simply needs the correct weight vector and a solver of choice. For a Non\textendash FC-family, the reader needs the UC families which yield the infeasible system of inequalities and the Farkas duals.} 

In the next definition, we may assume that $U(\mathcal{S})=U(\mathcal{F})=U(\mathcal{A})=[n]$, for some positive integer $n$.
\begin{definition} A family of sets $\mathcal{S}$ generates $\mathcal{F}$ with a UC family $\mathcal{A}$, denoted by $\gen{\mathcal{S}}_{\mathcal{A}}:=\mathcal{F}$, if and only if $\mathcal{F}$ is a UC family that contains $\mathcal{S}$ such that $\mathcal{F} \uplus \mathcal{A} = \mathcal{F}$, and there exists no UC family $\widetilde{\mathcal{F}}\subset\mathcal{F}$ such that $\widetilde{\mathcal{F}}$ contains $\mathcal{S}$ and $\widetilde{\mathcal{F}} \uplus \mathcal{A} = \widetilde{\mathcal{F}}$.
\end{definition}
 As in the previous Section, for all UC families $\mathcal{A}$ that are ``candidate" FC-families in the following propositions and definition, we assume that $U(\mathcal{A})=[n]$, for some integer $n\geq 1$.
\begin{prop}[FC inequalities] \label{FCineq}
Let $\mathcal{A}$ be a UC family such that $\emptyset \in \mathcal{A}$, and let $c \in \mathbb{Z}^n_{\geq 0}$ such that $\sum_{i\in [n]}c_i \geq 1$. Let $S \in \mathcal{A}$, and let $U,T \in \mathcal{P}([n])$ such that $S \cup U = F$ and $S \cup T = F$. Then the following
\begin{equation}
x_T+x_U - x_{T\cup U} - x_F \leq 0 \nonumber,
\end{equation}
is valid for $X(\mathcal{A},c)$.
\end{prop}
\begin{proof}
Suppose there exists an integer vector in $X(\mathcal{A},c)$ which yields a UC family $\mathcal{F}$ such that the following inequality holds (for some $S \in \mathcal{A}$ and $U,T \in \mathcal{P}([n])$ as above)
\begin{equation}
x_T+x_U - x_{T\cup U} - x_F \geq 1 \nonumber.
\end{equation}
This implies that the number of variables which equal one with positive coefficients is greater than the number of variables with negative coefficients which equal one. But if either $x_T$ or $x_U$ are one then $x_F$ is one (if both are one then $x_{T\cup U}$ is one) and we arrive at a contradiction. 
\end{proof}
In the following definition the role of a considered UC family $\mathcal{A}$ is taken into account in the listed conditions. In the first condition the role of $\mathcal{A}$ is implicit in the existence of a FS inequality, whereas in the second condition the role of $\mathcal{A}$ is implicit in \emph{generating} the desired family, as discussed at the beginning of this section.
\begin{definition}[FC-chain] \label{fcchain}
Let $\mathcal{A}$ be a UC family such that $\emptyset \in \mathcal{A}$, and let $c \in \mathbb{Z}^n_{\geq 0}$ such that $\sum_{i\in [n]}c_i \geq 1$. Let $\mathcal{S},\mathcal{S}'\subset \mathcal{P}([n])$, $\mathcal{S} \cap \mathcal{S}' = \emptyset$. Given $B_i \in \mathcal{S}, B_j \in \mathcal{S}'$, we say $B_i,B_j$ form an FC-chain which we denote by $B_i \longrightarrow B_j$, if and only if there exist tuples $(B_i,B_k),(B_k,B_l),(B_l,B_m), \ldots ,(B_p,B_j)$, where $\left\{B_k,B_l, \ldots ,B_p\right\} \subset \mathcal{P}([n])$, such that for any tuple $(B_q,B_r)$ in the FC-chain, at least one of the following conditions holds:
\begin{enumerate}
\item There exists $A \in \mathcal{A}$ such that $A \cup B_q = B_r$, and therefore $x_{B_q} \leq x_{B_{r}}$ is a valid FS inequality for $X(\mathcal{A},c)$.
\item There exists $S \in \gen{\mathcal{S}}_{\mathcal{A}}$ such that $x_{B_q} + x_{S} \leq 1 + x_{B_{r}}$ is a valid UC inequality for $X(\mathcal{A},c)$.
\end{enumerate}
\end{definition}
The following proposition follows directly from the definition above.
\begin{prop}[FC-chain inequalities] \label{FCineq}
Let $\mathcal{A}$ be a UC family such that $\emptyset \in \mathcal{A}$, and let $c \in \mathbb{Z}^n_{\geq 0}$ such that $\sum_{i\in [n]}c_i \geq 1$. Let $\mathcal{S},\mathcal{S}'\subset \mathcal{P}([n])$, $\mathcal{S} \cap \mathcal{S}' = \emptyset $. For any $\mathcal{T} \subseteq \mathcal{S}$ define $\mathcal{U}(\mathcal{T}):=\left\{S' \in \mathcal{S}' \ | \ \exists \ S \in \mathcal{T}: S \longrightarrow S'\right\}$. Suppose that $|\mathcal{T}| \leq |\mathcal{U}(\mathcal{T})|$ for all $\mathcal{T} \subseteq \mathcal{S}$. Then the
inequality
\begin{equation}
\sum_{S \in \mathcal{S}}x_S - \sum_{S \in \mathcal{S}'}x_S \leq 0 \nonumber,
\end{equation}
is valid for $X(\mathcal{A},c)$.
\end{prop}
\begin{proof}
Suppose there exists an integer vector in $X(\mathcal{A},c)$ which yields a UC family $\mathcal{F}$ such that the following inequality holds (for some $\mathcal{S},\mathcal{S}'\subset \mathcal{P}([n])$, as above)
\begin{equation}
\sum_{S \in \mathcal{S}\cap \mathcal{F}}x_S - \sum_{S \in \mathcal{S}'\cap \mathcal{F}}x_S \geq 1 \nonumber.
\end{equation}
It is clear that $\mathcal{S}\cap \mathcal{F} \neq \emptyset$, otherwise we arrive at a contradiction. Therefore the inequality implies that the number of variables $x_S$ which equal one, for all $S \in \mathcal{S}\cap \mathcal{F}$ is greater than the number of variables $x_S$ which equal one, for all $S \in \mathcal{S}'\cap \mathcal{F}$. Let $\mathcal{T} \subseteq \mathcal{S}\cap \mathcal{F}$, and for all $S \in \mathcal{T}$, let $x_S=1$. $|\mathcal{T}| \leq |\mathcal{U}(\mathcal{T})|$ holds by hypothesis. Furthermore by the definition of an FC-chain for each $\mathcal{T},\mathcal{S}' \subset \mathcal{P}([n])$ such that $\mathcal{T}\cap \mathcal{S}' = \emptyset$, for all $S' \in \mathcal{U}(\mathcal{T})$ we conclude that $x_{S'}=1$. Thus we arrive at a contradiction.
\end{proof}
Observe that FC-chain inequalities generalize FC-inequalities. We will use them in the appendix to explicitly exhibit the branch and bound tree of the counterexample in the next section. In particular, this implies that our counterexample requires no \emph{trust} from the reader, in the sense that its verification can be separated from the complex optimization process that produced it.
\section{Generators for Non\textendash FC-families}\label{sec:Generators}
In this section we exhibit a counterexample to a conjecture of Morris \cite{Morris} about generators for Non-FC-families.
\begin{definition}[regular] \label{regular}
Let $\mathcal{S}$ be a family of sets such that $U(\mathcal{S})=[n]$. Suppose $\mathcal{S}$ is a minimal generator for a UC family $\mathcal{F}$, such that $\mathcal{F}$ is a Non\textendash FC-family. Then $\mathcal{S}$ is \emph{regular} if and only if for any $A \in \mathcal{S}$, $A \neq \emptyset$, and any $i \in [n]$, the UC family 
$\gen{(\mathcal{S}\setminus \left\{A\right\})\cup \left\{A \cup \left\{i\right\} \right\}}$ is Non\textendash FC.
\end{definition}
\begin{conj}[Morris 2006] \label{conjecture}
Let $\mathcal{S}$ be a family of sets such that $U(\mathcal{S})=[n]$, for $n \geq 3$. Suppose $\mathcal{S}$ is a minimal generator for a UC family $\mathcal{F}$, such that $\mathcal{F}$ is a Non\textendash FC-family. Then $\mathcal{S}$ is regular.
\end{conj}
Morris \cite{Morris} checked the conjecture for all known families at the time, and therefore considered it plausible. In some sense, Conjecture \ref{conjecture} perfectly illustrates our general lack of knowledge about UC families since\textemdash as a number of other related questions\textemdash it has eluded an answer for a relatively long time. The obstacle\textemdash in this case and others to follow\textemdash is the lack of a method for exactly characterizing FC-families, a gap in knowledge which we correct with our framework. 
\subsection{A Counterexample for Structures in Non\textendash FC-families}
Our counterexample on six elements is minimal, in the sense that Morris \cite{Morris} completely characterizes FC-families on 5 elements. \\ Let $\mathcal{S}:=\left\{\emptyset, \left\{4,5,6 \right\}, \left\{1,3,4 \right\}, \left\{1,2,5,6 \right\}, \left\{1,2,3,4 \right\} \right\} \subset \mathcal{P}([6])$. Furthermore, let $\mathcal{T}:=\{\left\{1,2,4,5,6 \right\}, \left\{1,3,4,5,6 \right\}, \left\{1,2,3,4,5,6 \right\} \} \subset \mathcal{P}([6])$. Hence it follows that $\gen{\mathcal{S}} = \mathcal{S} \cup \mathcal{T}$. It is straightforward to check that $\mathcal{S}$ is a minimal generator for $\mathcal{S} \cup \mathcal{T}$. We will show that $\gen{\mathcal{S}}$ is a Non\textendash FC-family. There is a stronger connection between the structure of inequalities featured in the proof below and questions of Vaughan \cite{Vaughan2} and Morris \cite{Morris} we answer later in this work. In Section \ref{sec:Relaxation} we explicitly describe the structure of UC families from which the inequalities below are derived in relation to the questions of interest. 
\begin{prop}\label{nonfcounter}
$\gen{\mathcal{S}}$ is a Non\textendash FC-family.
\end{prop}

\begin{proof}
Algorithm \ref{row generation} determines an infeasible system of constraints which yields the result. We display an irreducible infeasible subset of the given system. We identify columns with zero one entries for each $S \in \mathcal{P}([6])$. The six matrices featured below represent UC families. The top row keeps track of the number of sets in each family. In addition to rechecking with an exact rational solver \cite{Exact} and other solvers, we check that each matrix is UC via simple external subroutines and finally by hand. Furthermore, let $\mathcal{F}\subset \mathcal{P}([6])$ be a family represented by one of the matrices below. By inspection we see that $\mathcal{F} \uplus \gen{\mathcal{S}} = \mathcal{F}$. In each matrix, we color columns which correspond to sets in $\mathcal{S}$, $\mathcal{T}$, red and blue, respectively. Each matrix yields an Inequality~\eqref{poon} from Poonen's Theorem (multiplied by two) featured below it. The following system of constraints is infeasible in nonnegative $y_i$ for all $1 \leq i \leq 6$. For each row we display the Farkas dual values in square brackets. This yields a certificate of infeasibility via a straightforward application of Farkas' Lemma. For convenience we state the lemma in the appendix.
\\
$\tt{[-7190]:y_1 + y_2 + y_3  + y_4 + y_5 + y_6 = 1}$.
\\
\vspace{4mm}
\[
\begin{tikzpicture}[ transform canvas={scale=0.51} ]
\matrix (m)[matrix of math nodes]
{   \color{red}{0} &  \color{red}{1} &  \color{red}{2} &  \color{red}{3} &  \color{red}{4} &  \color{red}{5} &  \color{red}{6} &  \color{red}{7} &  \color{red}{8} &  \color{red}{9} &  \color{red}{10} &  \color{red}{11} &  \color{red}{12} &  \color{red}{13} &  \color{red}{14} &  \color{red}{15} &  \color{red}{16} &  \color{red}{17} &  \color{red}{18} &  \color{red}{19} &  \color{red}{20} &  \color{red}{21} &  \color{red}{22} &  \color{red}{23} &  \color{red}{24} &  \color{red}{25} &  \color{red}{26} &  \color{red}{27} &  \color{red}{28} &  \color{red}{29} &  \color{red}{30} &  \color{red}{31} &  \color{red}{32} &  \color{red}{33} &  \color{red}{34} &  \color{red}{35} &  \color{red}{36} &  \color{red}{37} &  \color{red}{38} &  \color{red}{39} &  \color{red}{40} &  \color{red}{41} &  \color{red}{42} &  \color{red}{43} \\ 
c_1& \color{red}{0}& 0& 0& 0& 0& 0& 0& \color{red}{0}& 0& 0& 0& 0& 0& 0& 0& 0& 0& 0& 0& 0& 0& 0& 0& 0& 0& 0& 0& 0& 0& 0& 0& 0& \color{red}{1}& 1& 1& \color{blue}{1}& \color{red}{1}& \color{blue}{1}& 1& \color{red}{1}& 1& 1& \color{blue}{1}\\ 
c_2& \color{red}{0}& 0& 0& 0& 0& 0& 0& \color{red}{0}& 0& 0& 0& 0& 0& 0& 0& 0& 1& 1& 1& 1& 1& 1& 1& 1& 1& 1& 1& 1& 1& 1& 1& 1& \color{red}{0}& 0& 0& \color{blue}{0}& \color{red}{1}& \color{blue}{1}& 1& \color{red}{1}& 1& 1& \color{blue}{1}\\ 
c_3& \color{red}{0}& 0& 0& 0& 0& 0& 0& \color{red}{0}& 1& 1& 1& 1& 1& 1& 1& 1& 0& 0& 0& 0& 0& 0& 0& 0& 1& 1& 1& 1& 1& 1& 1& 1& \color{red}{1}& 1& 1& \color{blue}{1}& \color{red}{0}& \color{blue}{0}& 1& \color{red}{1}& 1& 1& \color{blue}{1}\\ 
c_4& \color{red}{0}& 0& 0& 0& 1& 1& 1& \color{red}{1}& 0& 0& 0& 0& 1& 1& 1& 1& 0& 0& 0& 0& 1& 1& 1& 1& 0& 0& 0& 0& 1& 1& 1& 1& \color{red}{1}& 1& 1& \color{blue}{1}& \color{red}{0}& \color{blue}{1}& 0& \color{red}{1}& 1& 1& \color{blue}{1}\\ 
c_5& \color{red}{0}& 0& 1& 1& 0& 0& 1& \color{red}{1}& 0& 0& 1& 1& 0& 0& 1& 1& 0& 0& 1& 1& 0& 0& 1& 1& 0& 0& 1& 1& 0& 0& 1& 1& \color{red}{0}& 0& 1& \color{blue}{1}& \color{red}{1}& \color{blue}{1}& 1& \color{red}{0}& 0& 1& \color{blue}{1}\\ 
c_6& \color{red}{0}& 1& 0& 1& 0& 1& 0& \color{red}{1}& 0& 1& 0& 1& 0& 1& 0& 1& 0& 1& 0& 1& 0& 1& 0& 1& 0& 1& 0& 1& 0& 1& 0& 1& \color{red}{0}& 1& 0& \color{blue}{1}& \color{red}{1}& \color{blue}{1}& 1& \color{red}{0}& 1& 0& \color{blue}{1}\\ 
};
\end{tikzpicture}
\]
\vspace{5mm}
\\
$\tt{[30]:22y_1 + 46y_2 + 50y_3  + 50y_4 + 46y_5 + 46y_6 \geq 43}$.
\vspace{5mm}
\\
\[
\begin{tikzpicture}[ transform canvas={scale=0.57} ]
\matrix (m)[matrix of math nodes]
{  \color{red}{0} &  \color{red}{1} &  \color{red}{2} &  \color{red}{3} &  \color{red}{4} &  \color{red}{5} &  \color{red}{6} &  \color{red}{7} &  \color{red}{8} &  \color{red}{9} &  \color{red}{10} &  \color{red}{11} &  \color{red}{12} &  \color{red}{13} &  \color{red}{14} &  \color{red}{15} &  \color{red}{16} &  \color{red}{17} &  \color{red}{18} &  \color{red}{19} &  \color{red}{20} &  \color{red}{21} &  \color{red}{22} &  \color{red}{23} &  \color{red}{24} &  \color{red}{25} &  \color{red}{26} &  \color{red}{27} &  \color{red}{28} &  \color{red}{29} &  \color{red}{30} &  \color{red}{31} &  \color{red}{32} &  \color{red}{33} &  \color{red}{34} &  \color{red}{35} &  \color{red}{36} &  \color{red}{37} &  \color{red}{38} &  \color{red}{39} \\ 
c_1& \color{red}{0}& 0& 0& 0& 0& 0& 0& \color{red}{0}& 0& 0& 0& 0& 0& 0& 0& 0& 1& 1& 1& 1& 1& 1& 1& 1& 1& 1& 1& 1& \color{red}{1}& 1& 1& \color{blue}{1}& \color{red}{1}& \color{blue}{1}& 1& \color{red}{1}& 1& 1& \color{blue}{1}\\ 
c_2& \color{red}{0}& 0& 0& 0& 0& 0& 0& \color{red}{0}& 0& 0& 0& 0& 0& 0& 0& 0& 0& 0& 0& 0& 0& 0& 0& 0& 0& 0& 0& 0& \color{red}{0}& 0& 0& \color{blue}{0}& \color{red}{1}& \color{blue}{1}& 1& \color{red}{1}& 1& 1& \color{blue}{1}\\ 
c_3& \color{red}{0}& 0& 0& 0& 0& 0& 0& \color{red}{0}& 1& 1& 1& 1& 1& 1& 1& 1& 0& 0& 0& 0& 0& 0& 0& 0& 1& 1& 1& 1& \color{red}{1}& 1& 1& \color{blue}{1}& \color{red}{0}& \color{blue}{0}& 1& \color{red}{1}& 1& 1& \color{blue}{1}\\ 
c_4& \color{red}{0}& 0& 0& 0& 1& 1& 1& \color{red}{1}& 0& 0& 0& 0& 1& 1& 1& 1& 0& 0& 0& 0& 1& 1& 1& 1& 0& 0& 0& 0& \color{red}{1}& 1& 1& \color{blue}{1}& \color{red}{0}& \color{blue}{1}& 0& \color{red}{1}& 1& 1& \color{blue}{1}\\ 
c_5& \color{red}{0}& 0& 1& 1& 0& 0& 1& \color{red}{1}& 0& 0& 1& 1& 0& 0& 1& 1& 0& 0& 1& 1& 0& 0& 1& 1& 0& 0& 1& 1& \color{red}{0}& 0& 1& \color{blue}{1}& \color{red}{1}& \color{blue}{1}& 1& \color{red}{0}& 0& 1& \color{blue}{1}\\ 
c_6& \color{red}{0}& 1& 0& 1& 0& 1& 0& \color{red}{1}& 0& 1& 0& 1& 0& 1& 0& 1& 0& 1& 0& 1& 0& 1& 0& 1& 0& 1& 0& 1& \color{red}{0}& 1& 0& \color{blue}{1}& \color{red}{1}& \color{blue}{1}& 1& \color{red}{0}& 1& 0& \color{blue}{1}\\ 
};

\end{tikzpicture}
\]
\\
\vspace{5mm}
\\
$\tt{[9]:46y_1 + 14y_2 + 42y_3  + 42y_4 + 42y_5 + 42y_6 \geq 39}$.
\\
\vspace{5mm}
\\
\[
\begin{tikzpicture}[ transform canvas={scale=0.48} ]
\matrix (m)[matrix of math nodes]
{  \color{red}{0} &  \color{red}{1} &  \color{red}{2} &  \color{red}{3} &  \color{red}{4} &  \color{red}{5} &  \color{red}{6} &  \color{red}{7} &  \color{red}{8} &  \color{red}{9} &  \color{red}{10} &  \color{red}{11} &  \color{red}{12} &  \color{red}{13} &  \color{red}{14} &  \color{red}{15} &  \color{red}{16} &  \color{red}{17} &  \color{red}{18} &  \color{red}{19} &  \color{red}{20} &  \color{red}{21} &  \color{red}{22} &  \color{red}{23} &  \color{red}{24} &  \color{red}{25} &  \color{red}{26} &  \color{red}{27} &  \color{red}{28} &  \color{red}{29} &  \color{red}{30} &  \color{red}{31} &  \color{red}{32} &  \color{red}{33} &  \color{red}{34} &  \color{red}{35} &  \color{red}{36} &  \color{red}{37} &  \color{red}{38} &  \color{red}{39} &  \color{red}{40} &  \color{red}{41} &  \color{red}{42} &  \color{red}{43} &  \color{red}{44} &  \color{red}{45} &  \color{red}{46} \\ 
c_1& \color{red}{0}& 0& 0& 0& \color{red}{0}& 0& 0& 0& 0& 0& 0& 0& 0& 0& 0& 0& 0& 0& 0& 0& 1& 1& 1& 1& 1& 1& 1& 1& 1& \color{red}{1}& 1& 1& \color{blue}{1}& 1& 1& 1& \color{red}{1}& \color{blue}{1}& 1& 1& 1& 1& \color{red}{1}& 1& 1& \color{blue}{1}\\ 
c_2& \color{red}{0}& 0& 0& 0& \color{red}{0}& 0& 0& 0& 0& 0& 1& 1& 1& 1& 1& 1& 1& 1& 1& 1& 0& 0& 0& 0& 0& 0& 0& 0& 0& \color{red}{0}& 0& 0& \color{blue}{0}& 1& 1& 1& \color{red}{1}& \color{blue}{1}& 1& 1& 1& 1& \color{red}{1}& 1& 1& \color{blue}{1}\\ 
c_3& \color{red}{0}& 0& 0& 0& \color{red}{0}& 1& 1& 1& 1& 1& 0& 0& 0& 0& 0& 1& 1& 1& 1& 1& 0& 0& 0& 0& 0& 1& 1& 1& 1& \color{red}{1}& 1& 1& \color{blue}{1}& 0& 0& 0& \color{red}{0}& \color{blue}{0}& 1& 1& 1& 1& \color{red}{1}& 1& 1& \color{blue}{1}\\ 
c_4& \color{red}{0}& 0& 0& 0& \color{red}{1}& 0& 0& 0& 0& 1& 0& 0& 0& 0& 1& 0& 0& 0& 0& 1& 0& 0& 0& 0& 1& 0& 0& 0& 0& \color{red}{1}& 1& 1& \color{blue}{1}& 0& 0& 0& \color{red}{0}& \color{blue}{1}& 0& 0& 0& 0& \color{red}{1}& 1& 1& \color{blue}{1}\\ 
c_5& \color{red}{0}& 0& 1& 1& \color{red}{1}& 0& 0& 1& 1& 1& 0& 0& 1& 1& 1& 0& 0& 1& 1& 1& 0& 0& 1& 1& 1& 0& 0& 1& 1& \color{red}{0}& 0& 1& \color{blue}{1}& 0& 0& 1& \color{red}{1}& \color{blue}{1}& 0& 0& 1& 1& \color{red}{0}& 0& 1& \color{blue}{1}\\ 
c_6& \color{red}{0}& 1& 0& 1& \color{red}{1}& 0& 1& 0& 1& 1& 0& 1& 0& 1& 1& 0& 1& 0& 1& 1& 0& 1& 0& 1& 1& 0& 1& 0& 1& \color{red}{0}& 1& 0& \color{blue}{1}& 0& 1& 0& \color{red}{1}& \color{blue}{1}& 0& 1& 0& 1& \color{red}{0}& 1& 0& \color{blue}{1}\\ 
};

\end{tikzpicture}
\]
\\
\vspace{5mm}
\\
$\tt{[44]:52y_1 + 46y_2 + 52y_3  + 28y_4 + 52y_5 + 52y_6 \geq 46}$.
\\
\vspace{5mm}
\\
\[
\begin{tikzpicture}[ transform canvas={scale=0.56} ]
\matrix (m)[matrix of math nodes]
{   \color{red}{0} &  \color{red}{1} &  \color{red}{2} &  \color{red}{3} &  \color{red}{4} &  \color{red}{5} &  \color{red}{6} &  \color{red}{7} &  \color{red}{8} &  \color{red}{9} &  \color{red}{10} &  \color{red}{11} &  \color{red}{12} &  \color{red}{13} &  \color{red}{14} &  \color{red}{15} &  \color{red}{16} &  \color{red}{17} &  \color{red}{18} &  \color{red}{19} &  \color{red}{20} &  \color{red}{21} &  \color{red}{22} &  \color{red}{23} &  \color{red}{24} &  \color{red}{25} &  \color{red}{26} &  \color{red}{27} &  \color{red}{28} &  \color{red}{29} &  \color{red}{30} &  \color{red}{31} &  \color{red}{32} &  \color{red}{33} &  \color{red}{34} &  \color{red}{35} &  \color{red}{36} &  \color{red}{37} &  \color{red}{38} &  \color{red}{39} &  \color{red}{40} \\ 
c_1& \color{red}{0}& 0& 0& 0& 0& 0& 0& \color{red}{0}& 0& 0& 0& 0& 0& 0& 0& 0& 1& 1& 1& 1& 1& 1& 1& 1& \color{red}{1}& 1& 1& \color{blue}{1}& 1& 1& 1& \color{red}{1}& 1& 1& 1& \color{blue}{1} & \color{red}{1}& 1& 1& \color{blue}{1}\\ 
c_2& \color{red}{0}& 0& 0& 0& 0& 0& 0& \color{red}{0}& 1& 1& 1& 1& 1& 1& 1& 1& 0& 0& 0& 0& 0& 0& 0& 0& \color{red}{0}& 0& 0& \color{blue}{0}& 1& 1& 1& \color{red}{1}& 1& 1& 1& \color{blue}{1}& \color{red}{1}& 1& 1& \color{blue}{1}\\ 
c_3& \color{red}{0}& 0& 0& 0& 0& 0& 0& \color{red}{0}& 0& 0& 0& 0& 0& 0& 0& 0& 0& 0& 0& 0& 0& 0& 0& 0& \color{red}{1}& 1& 1& \color{blue}{1}& 0& 0& 0& \color{red}{0}& 0& 0& 0& \color{blue}{0}& \color{red}{1}& 1& 1& \color{blue}{1}\\ 
c_4& \color{red}{0}& 0& 0& 0& 1& 1& 1& \color{red}{1}& 0& 0& 0& 0& 1& 1& 1& 1& 0& 0& 0& 0& 1& 1& 1& 1& \color{red}{1}& 1& 1& \color{blue}{1}& 0& 0& 0& \color{red}{0}& 1& 1& 1& \color{blue}{1}& \color{red}{1}& 1& 1& \color{blue}{1}\\ 
c_5& \color{red}{0}& 0& 1& 1& 0& 0& 1& \color{red}{1}& 0& 0& 1& 1& 0& 0& 1& 1& 0& 0& 1& 1& 0& 0& 1& 1& \color{red}{0}& 0& 1& \color{blue}{1}& 0& 0& 1& \color{red}{1}& 0& 0& 1& \color{blue}{1}& \color{red}{0}& 0& 1& \color{blue}{1}\\ 
c_6& \color{red}{0}& 1& 0& 1& 0& 1& 0& \color{red}{1}& 0& 1& 0& 1& 0& 1& 0& 1& 0& 1& 0& 1& 0& 1& 0& 1& \color{red}{0}& 1& 0& \color{blue}{1}& 0& 1& 0& \color{red}{1}& 0& 1& 0& \color{blue}{1}& \color{red}{0}& 1& 0& \color{blue}{1}\\ 
};

\end{tikzpicture}
\]
\\
\vspace{5mm}
\\
$\tt{[21]:48y_1 + 40y_2 + 16y_3  + 48y_4 + 40y_5 + 40y_6 \geq 40}$.
\\
\vspace{5mm}
\\
\[
\begin{tikzpicture}[ transform canvas={scale=0.53} ]
\matrix (m)[matrix of math nodes]
{  \color{red}{0} &  \color{red}{1} &  \color{red}{2} &  \color{red}{3} &  \color{red}{4} &  \color{red}{5} &  \color{red}{6} &  \color{red}{7} &  \color{red}{8} &  \color{red}{9} &  \color{red}{10} &  \color{red}{11} &  \color{red}{12} &  \color{red}{13} &  \color{red}{14} &  \color{red}{15} &  \color{red}{16} &  \color{red}{17} &  \color{red}{18} &  \color{red}{19} &  \color{red}{20} &  \color{red}{21} &  \color{red}{22} &  \color{red}{23} &  \color{red}{24} &  \color{red}{25} &  \color{red}{26} &  \color{red}{27} &  \color{red}{28} &  \color{red}{29} &  \color{red}{30} &  \color{red}{31} &  \color{red}{32} &  \color{red}{33} &  \color{red}{34} &  \color{red}{35} &  \color{red}{36} &  \color{red}{37} &  \color{red}{38} &  \color{red}{39} &  \color{red}{40} &  \color{red}{41} &  \color{red}{42} \\ 
c_1& \color{red}{0}& 0& 0& 0& \color{red}{0}& 0& 0& 0& 0& 0& 0& 0& 0& 0& 0& 0& 0& 0& 0& 0& 1& 1& 1& 1& 1& 1& 1& \color{red}{1}& 1& \color{blue}{1}& 1& 1& \color{red}{1}& 1& 1& \color{blue}{1}& 1& 1& 1& \color{red}{1}& 1& \color{blue}{1}\\ 
c_2& \color{red}{0}& 0& 0& 0& \color{red}{0}& 0& 0& 0& 0& 0& 1& 1& 1& 1& 1& 1& 1& 1& 1& 1& 0& 0& 0& 0& 0& 0& 0& \color{red}{0}& 0& \color{blue}{0}& 1& 1& \color{red}{1}& 1& 1& \color{blue}{1}& 1& 1& 1& \color{red}{1}& 1& \color{blue}{1}\\ 
c_3& \color{red}{0}& 0& 0& 0& \color{red}{0}& 1& 1& 1& 1& 1& 0& 0& 0& 0& 0& 1& 1& 1& 1& 1& 0& 0& 0& 0& 0& 1& 1& \color{red}{1}& 1& \color{blue}{1}& 0& 0& \color{red}{0}& 0& 0& \color{blue}{0}& 1& 1& 1& \color{red}{1}& 1& \color{blue}{1}\\ 
c_4& \color{red}{0}& 0& 1& 1& \color{red}{1}& 0& 0& 1& 1& 1& 0& 0& 1& 1& 1& 0& 0& 1& 1& 1& 0& 0& 1& 1& 1& 0& 0& \color{red}{1}& 1& \color{blue}{1}& 0& 0& \color{red}{0}& 1& 1& \color{blue}{1}& 0& 0& 0& \color{red}{1}& 1& \color{blue}{1}\\ 
c_5& \color{red}{0}& 0& 0& 0& \color{red}{1}& 0& 0& 0& 0& 1& 0& 0& 0& 0& 1& 0& 0& 0& 0& 1& 0& 0& 0& 0& 1& 0& 0& \color{red}{0}& 0& \color{blue}{1}& 0& 0& \color{red}{1}& 0& 0& \color{blue}{1}& 0& 0& 1& \color{red}{0}& 0& \color{blue}{1}\\ 
c_6& \color{red}{0}& 1& 0& 1& \color{red}{1}& 0& 1& 0& 1& 1& 0& 1& 0& 1& 1& 0& 1& 0& 1& 1& 0& 1& 0& 1& 1& 0& 1& \color{red}{0}& 1& \color{blue}{1}& 0& 1& \color{red}{1}& 0& 1& \color{blue}{1}& 0& 1& 1& \color{red}{0}& 1& \color{blue}{1}\\ 
};

\end{tikzpicture}
\]
\\
\vspace{5mm}
\\
$\tt{[32]:44y_1 + 44y_2 + 42y_3  + 48y_4 + 20y_5 + 52y_6 \geq 42}$.
\\
\vspace{5mm}
\\
\[
\begin{tikzpicture}[ transform canvas={scale=0.53} ]
\matrix (m)[matrix of math nodes]
{  \color{red}{0} &  \color{red}{1} &  \color{red}{2} &  \color{red}{3} &  \color{red}{4} &  \color{red}{5} &  \color{red}{6} &  \color{red}{7} &  \color{red}{8} &  \color{red}{9} &  \color{red}{10} &  \color{red}{11} &  \color{red}{12} &  \color{red}{13} &  \color{red}{14} &  \color{red}{15} &  \color{red}{16} &  \color{red}{17} &  \color{red}{18} &  \color{red}{19} &  \color{red}{20} &  \color{red}{21} &  \color{red}{22} &  \color{red}{23} &  \color{red}{24} &  \color{red}{25} &  \color{red}{26} &  \color{red}{27} &  \color{red}{28} &  \color{red}{29} &  \color{red}{30} &  \color{red}{31} &  \color{red}{32} &  \color{red}{33} &  \color{red}{34} &  \color{red}{35} &  \color{red}{36} &  \color{red}{37} &  \color{red}{38} &  \color{red}{39} &  \color{red}{40} &  \color{red}{41} &  \color{red}{42} \\ 
c_1& \color{red}{0}& 0& 0& 0& \color{red}{0}& 0& 0& 0& 0& 0& 0& 0& 0& 0& 0& 0& 0& 0& 0& 0& 1& 1& 1& 1& 1& 1& 1& \color{red}{1}& 1& \color{blue}{1}& 1& 1& \color{red}{1}& 1& 1& \color{blue}{1}& 1& 1& 1& \color{red}{1}& 1& \color{blue}{1}\\ 
c_2& \color{red}{0}& 0& 0& 0& \color{red}{0}& 0& 0& 0& 0& 0& 1& 1& 1& 1& 1& 1& 1& 1& 1& 1& 0& 0& 0& 0& 0& 0& 0& \color{red}{0}& 0& \color{blue}{0}& 1& 1& \color{red}{1}& 1& 1& \color{blue}{1}& 1& 1& 1& \color{red}{1}& 1& \color{blue}{1}\\ 
c_3& \color{red}{0}& 0& 0& 0& \color{red}{0}& 1& 1& 1& 1& 1& 0& 0& 0& 0& 0& 1& 1& 1& 1& 1& 0& 0& 0& 0& 0& 1& 1& \color{red}{1}& 1& \color{blue}{1}& 0& 0& \color{red}{0}& 0& 0& \color{blue}{0}& 1& 1& 1& \color{red}{1}& 1& \color{blue}{1}\\ 
c_4& \color{red}{0}& 0& 1& 1& \color{red}{1}& 0& 0& 1& 1& 1& 0& 0& 1& 1& 1& 0& 0& 1& 1& 1& 0& 0& 1& 1& 1& 0& 0& \color{red}{1}& 1& \color{blue}{1}& 0& 0& \color{red}{0}& 1& 1& \color{blue}{1}& 0& 0& 0& \color{red}{1}& 1& \color{blue}{1}\\ 
c_5& \color{red}{0}& 1& 0& 1& \color{red}{1}& 0& 1& 0& 1& 1& 0& 1& 0& 1& 1& 0& 1& 0& 1& 1& 0& 1& 0& 1& 1& 0& 1& \color{red}{0}& 1& \color{blue}{1}& 0& 1& \color{red}{1}& 0& 1& \color{blue}{1}& 0& 1& 1& \color{red}{0}& 1& \color{blue}{1}\\ 
c_6& \color{red}{0}& 0& 0& 0& \color{red}{1}& 0& 0& 0& 0& 1& 0& 0& 0& 0& 1& 0& 0& 0& 0& 1& 0& 0& 0& 0& 1& 0& 0& \color{red}{0}& 0& \color{blue}{1}& 0& 0& \color{red}{1}& 0& 0& \color{blue}{1}& 0& 0& 1& \color{red}{0}& 0& \color{blue}{1}\\
};

\end{tikzpicture}
\]
\\
\vspace{5mm}
\\
$\tt{[32]:44y_1 + 44y_2 + 42y_3  + 48y_4 + 52y_5 + 20y_6 \geq 42}$.
\\

\end{proof}
We are now ready to show that $\mathcal{S}$ is not regular, and thus give a counterexample to Conjecture~\ref{conjecture}.
\begin{prop} \label{morriscounter}
 Let $\mathcal{S}' :=\left\{\emptyset, \left\{4,5,6 \right\}, \left\{1,3,4 \right\}, \left\{1,2,5,6 \right\}, \left\{1,2,3,4,5 \right\} \right\}$.
\\ Then $\gen{\mathcal{S}'}$ is an FC-family.
\end{prop}
\begin{proof}
Let $c \in \mathbb{Z}^6_{\geq 0}$ such that $c =(16,8,12,20,17,15)$. Then $IP(\gen{\mathcal{S}'},c)$ is infeasible \footnote{In the appendix we explicitly show the infeasibility of $IP(\gen{\mathcal{S}'},c)$ by making use of FC-chain inequalities and displaying irreducible infeasible subsets of constraints for the two leaf nodes of the resulting branch and bound tree.}. 
\end{proof}
\begin{cor}\label{counterexample}
$\mathcal{S}$ is a counterexample to Conjecture \ref{conjecture}.
\end{cor}
\begin{proof}
We show that $\mathcal{S}$ is not regular. We observe that $U(\mathcal{S})=[6]$ and $\mathcal{S}$ is a minimal generator for $\gen{\mathcal{S}}$. Furthermore from Proposition~\ref{nonfcounter} it follows that $\gen{\mathcal{S}}$ is a Non\textendash FC-family. However $\mathcal{S}' = (\mathcal{S}\setminus \{1,2,3,4\})\cup \left\{\{1,2,3,4\} \cup \{5\} \right\}$ and Proposition~\ref{morriscounter} implies that $\gen{\mathcal{S}'}$ is an FC-family.
\end{proof}
\iffalse We exhibit a few more new FC-families verified by our algorithm. As mentioned earlier, the computations are rechecked with a number of solvers followed by a final check with an exact rational solver and VIPR \cite{VIPR}. Furthermore, for all the families shown here, by adding FC-chain inequalities at the root node, it is possible to reduce the resulting branch and bound tree to a few nodes and check some irreducible infeasible subset of each leaf node (linear program) manually.
%\begin{proposition}
%Let $\mathcal{S} =\left\{ \left\{4,5,6,7 \right\}, \left\{1,3,4,5 \right\}, \left\{3,4,6,7 \right\}, \left\{5,6,7 \right\},  \left\{2,3,4 \right\} \right\}$.
%\\ Then $\gen{\mathcal{S}}$ is an FC-family.
%\end{proposition}
%\begin{proof}
%Let $\mathcal{C}=\{4,8,11,13,10,9,9\}$. Then $FC(\gen{\mathcal{S}},\mathcal{C})_{n}$ is infeasible. \qed
%\end{proof}

\begin{prop}
Let $\mathcal{S} =\left\{ \left\{4,5,6,7 \right\}, \left\{1,2,5,7 \right\}, \left\{3,4,5,7 \right\}, \left\{5,6,7 \right\},  \left\{4,6,7 \right\} \right\}$.
\\ Then $\gen{\mathcal{S}}$ is an FC-family.
\end{prop}
\begin{proof}
Let $\mathcal{C}=\{1,1,1,3,3,3,4\}$. Then $FC(\gen{\mathcal{S}},\mathcal{C})_{n}$ is infeasible. 
\end{proof}
\begin{prop}
Let $\mathcal{S} =\left\{ \left\{2,5,6,7 \right\}, \left\{2,4,6 \right\}, \left\{1,4,7,8 \right\}, \left\{1,2,3,7 \right\},  \left\{1,2,7,8 \right\} \right\}$.
\\ Then $\gen{\mathcal{S}}$ is an FC-family.
\end{prop}
\begin{proof}
Let $\mathcal{C}=\{10,15,5,13,3,13,12,4\}$. Then $FC(\gen{\mathcal{S}},\mathcal{C})_{n}$ is infeasible. 
\end{proof}
\fi
\section{Relaxation Questions}~\label{sec:Relaxation}
In this section, we briefly address the practical behavior of Algorithm~\ref{row generation}, as it sheds light on open questions of interests in Vaughan \cite{Vaughan2} and Morris \cite{Morris}. As a result, we exhibit a counterexample to the questions of Morris and Vaughan.

Our current implementation features $I^{\mathcal{A}}$ and $IP(\mathcal{A},c)$ in order to avoid possible numerical trouble by minimizing the sum of the $z_i$, in addition to selecting the ``sharpest cut" whenever we solve $IP(\mathcal{A},c)$. Yet, without witnessing first-hand computations for fixed UC families $\mathcal{A}$ such that $U(\mathcal{A})=[n]$ and $6 \leq n \leq 10$, Algorithm~\ref{row generation} may appear fraught with theoretical dangers.\footnote{ $IP(\mathcal{A},c)$ is a binary program with an exponential number of variables and constraints in $n$. Furthermore the number of iterations of Algorithm~\ref{row generation} could be exponential in $n$.} However, in practice our method is well-behaved in the described range, and is consequently the currently best available technique for the \emph{exact} determination of FC-families.% As mentioned inSection \ref{sec:Introduction}, ``mining" the given range we recover more exact characterization of minimal nonisomorphic FC-families than the total output of the past twenty-five years of research on the subject.

Furthermore, our implementation \emph{mostly} confirms the heuristic intuition of Vaughan and Morris as will be made explicit in the next paragraphs. Thus in the tested range, Algorithm~\ref{row generation} \emph{mostly} iterates $n$ times. However, in some cases it iterates more than $n$ (but less than $2n$) times\footnote{The runtimes vary roughly from a few seconds for $6\leq n \leq 7$ and a few minutes for $8 \leq n \leq 9$, to a few hours for $n=10$. Furthermore verification with exact SCIP \cite{Exact} takes longer, as does testing a non-minimal FC-family. Computations were carried out on machines with 2.40 GHz quad-core processors and 16 GB of RAM.}. Among the latter we find counterexamples to open questions of interest which we feature below.

As mentioned in the introduction, Vaughan \cite{Vaughan2} implements a heuristic that guides the search for a potential weight system. Given a UC family $\mathcal{A}$, $\emptyset \in \mathcal{A}$, the heuristic focuses only on UC families $\mathcal{B}$ with $\mathcal{B} \ \uplus \ \mathcal{A} = \mathcal{B}$, where $\mathcal{B} = \mathcal{P}([n] \setminus \left\{j\right\}) \uplus \mathcal{A}$ for all $j \in [n]$. If there exists a solution to the system of linear equations $\sum_{i\in [n]}y_i|\mathcal{B}_i| = |\mathcal{B}|/2$ in nonnegative $y_i$, with $\sum_{i\in [n]}y_i \leq 1$, then the considered UC family $\mathcal{A}$ becomes a candidate FC-family. All of Vaughan's candidate FC-families in \cite{Vaughan2} are identified as above, followed by tedious case analysis that spans several pages for the proof that the given family is FC. We precisely state Vaughan's question as follows:
\begin{quest}[Vaughan 2003] \label{vaughanq}
Let $\mathcal{A}$ be a UC family such that $U(\mathcal{A})=[n]$ and $\emptyset \in \mathcal{A}$. Consider UC families $\mathcal{B} \subseteq \mathcal{P}([n])$ such that $\mathcal{B} =  \mathcal{P}([n] \setminus \left\{j\right\}) \uplus \mathcal{A}$ for all $j \in [n]$. Suppose the linear system of equations $\sum_{i\in [n]}y_i|\mathcal{B}_i| = |\mathcal{B}|/2$ for all $\mathcal{B}$ as above has a solution in nonnegative reals $y_i$ for all $i \in [n]$, such that $\sum_{i\in [n]}y_i \leq 1$. Does this imply that $P^{\mathcal{A}}$ is nonempty?
\end{quest}
Given a UC family $\mathcal{A}$, $\emptyset \in \mathcal{A}$, Morris \cite{Morris} also focuses on $\mathcal{B}$ as above, searching instead for integer vectors contained in the polyhedron defined by the inequalities derived from the $n$ given $\mathcal{B}$ and $z_i \geq 0$ for all $i \in [n]$ with $\sum_{i\in [n]}z_i \geq 1$. The idea is that the $n$ given inequalities could capture information of interest without needing the rest of the possible inequalities. Morris shows that this holds in a number of cases, but is it true in general? More precisely, we state it as the following question:
\begin{quest}[Morris 2006]\label{morrisq}
Let $\mathcal{A}$ be a UC family such that $U(\mathcal{A})=[n]$ and $\emptyset \in \mathcal{A}$. Consider UC families $\mathcal{B} \subseteq \mathcal{P}([n])$ such that $\mathcal{B} =  \mathcal{P}([n] \setminus \left\{j\right\}) \uplus \mathcal{A}$ for all $j \in [n]$. Denote by $Z(\mathcal{A})$  the set of integer vectors contained in the polyhedron defined by $\sum_{i\in [n]}z_i \geq 1$, $\sum_{S\in \mathcal{B}}\left(\sum_{i \in S}z_i - \sum_{i \notin S}z_i \right) \geq 0$ for all $\mathcal{B}$ as above, and $ 0 \leq z_i$ for all $i \in [n]$. Suppose $Z(\mathcal{A})$ is nonempty. Does this imply that there exists a feasible solution of $I^{\mathcal{A}}$?
\end{quest}
Given a set $\mathcal{A}$ that yields a positive answer to Question~\ref{vaughanq}, we can scale the resulting vector $y$ and (after arbitrarily increasing some entries if necessary) arrive, following the proof of Corollary~\ref{integer}, at a vector $z$ that gives a positive answer to Question~\ref{morrisq}.
\begin{obs}\label{vaugmorr}
A positive answer to Question~\ref{vaughanq} for a given $\mathcal{A}$ implies a positive answer to Question~\ref{morrisq} for the same $\mathcal{A}$.
\end{obs}

Thus, considering the above, we can explicitly describe the structure associated with the Non\textendash FC-family that leads to the counterexample in Corollary~\ref{counterexample}. As above, it suffices to consider $\mathcal{B} \subseteq \mathcal{P}([n])$ such that $\mathcal{B} =  \mathcal{P}([n] \setminus \left\{j\right\}) \uplus \mathcal{A}$ for all $j \in [n]$, where $\mathcal{A}$ is our given UC family. This greatly simplifies the tedious task of checking that the algorithm's output is correct. Once the family is constructed according the given $\mathcal{B}$, it becomes straightforward to check that the necessary conditions for correctness are met. 

Given that the empty set does not make a difference in determining whether a UC family $\mathcal{A}$ is FC or Non\textendash FC, as we saw in Proposition~\ref{emptynot}, we may think the condition $\emptyset \in \mathcal{A}$ in the questions of Vaughan and Morris can be relaxed. If this were the case, the structure of the considered $\mathcal{B}$ with $\emptyset \not \in \mathcal{A}$ is again simplified, since the cardinality of the new family is at most the cardinality of the original one. Unfortunately, as we shall see, this is not the case. Still, in the next proposition, we show that a nonempty $Z(\mathcal{A})$ implies that a set of integer vectors contained in a polyhedron arising from ``smaller" structures is also nonempty.  

\begin{prop}\label{smallerstruct}
Let $\mathcal{A}$ be a UC family such that $U(\mathcal{A})=[n]$ and $\emptyset \in \mathcal{A}$. Suppose $Z(\mathcal{A})$ is nonempty. Consider $\mathcal{G} \subseteq \mathcal{P}([n])$ such that $\mathcal{G}=( \mathcal{P}([n] \setminus \left\{j\right\}) \uplus \mathcal{A}) \setminus  \mathcal{P}([n] \setminus \left\{j\right\})$ for all $j \in [n]$. Then the set of integer vectors contained in the polyhedron defined by $\sum_{i\in [n]}z_i \geq 1$, $\sum_{S\in \mathcal{G}}\left(\sum_{i \in S}z_i - \sum_{i \notin S}z_i \right) \geq 0$ for all $\mathcal{G}$ as above, and $ 0 \leq z_i$ for all $i \in [n]$, is nonempty. 
\end{prop}
\begin{proof}
Let $\mathcal{A}$ be a UC family such that $U(\mathcal{A})=[n]$ and $\emptyset \in \mathcal{A}$. Furthermore, let $\mathcal{B} \subseteq \mathcal{P}([n])$ such that $\mathcal{B} = \mathcal{P}([n] \setminus \left\{1\right\}) \uplus \mathcal{A}$. Since $\emptyset \in \mathcal{A}$, it follows that $ \mathcal{P}([n] \setminus \left\{1\right\})\subset \mathcal{B}$. Define $\mathcal{D}:=  \mathcal{P}([n] \setminus \left\{1\right\})$, $\mathcal{G}:= \mathcal{B} \setminus \mathcal{D}$. Suppose that $Z(\mathcal{A})$ is nonempty and $z \in Z(\mathcal{A})$. Define $\bar{z}$ as $z$ normalized by its $\ell_1$ norm. Thus we arrive at $\bar{z}_i \in \mathbb{Q}_{\geq 0}$ for all $i \in [n]$ and $ \sum_{i\in [n]} \bar{z}_i = 1$. Following the proof of Corollary~\ref{weight} we arrive at
\begin{flalign*}
\sum_{i\in [n]}2\bar{z}_i|\mathcal{B}_i| \geq |\mathcal{B}|
&\Longleftrightarrow \sum_{i\in [n]} 2\bar{z}_i|\mathcal{G}_i| + \sum_{i\in [n] \setminus \left\{1\right\}} 2\bar{z}_i|\mathcal{D}_i| \geq |\mathcal{G}| + |\mathcal{D}|\\
&\implies \sum_{i\in [n]} 2\bar{z}_i|\mathcal{G}_i| \geq |\mathcal{G}|.\\
\end{flalign*}\\
In the last implication we use $ \sum_{i\in [n] \setminus \left\{1\right\}} \bar{z}_i \leq 1$, with $\bar{z}_i \geq 0$ for all $i \in [n] \setminus \left\{1\right\}$. Furthermore, $\mathcal{D} = \mathcal{P}([n]\setminus \left\{1\right\})$ implies that $|\mathcal{D}_i|= 2^{n-2}$ for all $i \in [n] \setminus \left\{1\right\}$ and therefore
\begin{flalign*}
\sum_{i\in [n] \setminus \left\{1\right\}} 2\bar{z}_i|\mathcal{D}_i| &= |\mathcal{D}|\sum_{i\in [n] \setminus \left\{1\right\}}\bar{z}_i \leq |\mathcal{D}|.\\
\end{flalign*}\\
Since the same argument applies to $\mathcal{B} =  \mathcal{P}([n] \setminus \left\{j\right\}) \uplus \mathcal{A}$ for all $j \in [n]$, the desired result follows. 
\end{proof}
As we shall see next, a nonempty $Z(\mathcal{A}\setminus \left\{\emptyset\right\})$ does not necessarily imply a nonempty $Z(\mathcal{A})$.
\begin{prop}\label{revimp}
Let $\mathcal{A}$ be a UC family such that $U(\mathcal{A})=[n]$ and $\emptyset \in \mathcal{A}$. A nonempty $Z(\mathcal{A}\setminus \left\{\emptyset\right\})$ does not necessarily imply a nonempty $Z(\mathcal{A})$.
\end{prop} 
\begin{proof}
Let $\mathcal{S}:=\left\{ \emptyset, \left\{1,2,3\right\}, \left\{1,4,5\right\}, \left\{1,2,3,4\right\}, \left\{1,2,3,5\right\}, \left\{1,2,4,5\right\} \right\} \subset \mathcal{P}([5])$ and let $\widetilde{\mathcal{S}}:= \mathcal{S} \setminus  \left\{\emptyset\right\}$. Let $\mathcal{A}:=\gen{\mathcal{S}}$ and $\widetilde{\mathcal{A}}:= \gen{\widetilde{\mathcal{S}}}$. Morris \cite{Morris} proved that $Z(\mathcal{A})$ is empty. We show that $Z(\widetilde{\mathcal{A}})$ is nonempty. Observe that if we write each set in $\widetilde{\mathcal{A}}$ as a column of an $n\times m$ binary matrix $M$, we have more entries with ones than zeros. We conclude similarly for $\mathcal{B} \subseteq \mathcal{P}([n])$ such that $\mathcal{B} =  \mathcal{P}([n] \setminus \left\{j\right\}) \uplus \widetilde{\mathcal{A}}$ for all  $j \in [n]$. Hence, the (component-wise) all one vector is contained in $Z(\widetilde{\mathcal{A}})$. 
\end{proof}
\begin{cor}
The reverse implication in Proposition \ref{smallerstruct} does not necessarily hold.
\end{cor}
\begin{proof}
Follows directly from the proof of Proposition~\ref{revimp} where we exhibit an $\mathcal{A}$ such that $\emptyset \in \mathcal{A}$ and $Z(\mathcal{A})$ is empty. Then for each $j \in [n]$ we see that the binary matrix that represents $\mathcal{G}=( \mathcal{P}([n] \setminus \left\{j\right\}) \uplus \mathcal{A}) \setminus  \mathcal{P}([n] \setminus \left\{j\right\})$ has more entries with ones than zeros. 
\end{proof}
Finally, we give a negative answer to Morris' question, and also Vaughan's question.\vspace{0.75cm} \\ Let $\mathcal{S}:=\left\{\emptyset, \left\{2,3,4,6,7\right\}, \left\{1,2,3,4\right\}, \left\{1,3,4,6\right\}, \left\{5,6,7\right\}, \left\{3,4,7\right\} \right\} \subset \mathcal{P}([7])$. Furthermore, define $\mathcal{D}:=\gen{\mathcal{S}}$.
\begin{prop} \label{morrcount1}
$Z(\mathcal{D})$ is nonempty.
\end{prop}
\begin{proof}
We simply write down the relevant inequalities and exhibit a vector in $Z(\mathcal{D})$. The order of display matches $j$ in $\mathcal{B} =  \mathcal{P}([7] \setminus \left\{j\right\}) \uplus \mathcal{D}$ for each $j \in [7]$.
\begin{center}
\begin{tabular}{@{}l@{}}
$\tt{-52z_{1} +4z_{2} +12z_{3} +12z_{4} +4z_{6}\geq 0}$ \\
$\tt{+6z_{1} -54z_{2} +10z_{3} +10z_{4} +2z_{6} +2z_{7}\geq 0}$\\
$\tt{\color{red}{+6z_{1} +2z_{2} -42z_{3} +22z_{4} +2z_{6} +10z_{7}\geq 0}}$\\
$\tt{\color{red}{+6z_{1} +2z_{2} +22z_{3} -42z_{4} +2z_{6} +10z_{7}\geq 0}}$\\
$\tt{-48z_{5} +16z_{6} +16z_{7}\geq 0}$\\
$\tt{+5z_{1} +1z_{2} +7z_{3} +7z_{4} +13z_{5} -41z_{6} +15z_{7}\geq 0}$\\
$\tt{+12z_{3} +12z_{4} +12z_{5} +12z_{6} -36z_{7}\geq 0}$ \\

\end{tabular}
\end{center}
The vector $(7,5,12,12,10,14,16) \in \mathbb{Z}^7_{\geq 0}$ is contained in $Z(\mathcal{D})$.
\end{proof}
\begin{prop}\label{morrcount2}
$\mathcal{D}$ is a Non\textendash FC-family.
\end{prop}
\begin{proof}
Using Algorithm~\ref{row generation} we exhibit a system of linear inequalities that is infeasible and the result follows from Corollary~\ref{rqz}. As a certificate of infeasibility we display Farkas dual values in square brackets before each inequality. Structurally, we see that the only difference between the UC families that generated this system of linear inequalities and the previous one are the red inequalities. In contrast to the other inequalities, the red one here is derived from the following UC family: $( \mathcal{P}([7]\setminus \left\{3\right\} \setminus \left\{4\right\}) \uplus \mathcal{D}) \cup \left\{ \left\{1,3,4\right\},\left\{1,3,4,5 \right\} \right\}$.
\begin{center}
\begin{tabular}{@{}l@{}}
$\tt{[1]: z_{1} +z_{2} +z_{3} + z_{4} +z_{5} +z_{6}+z_{7}\geq 1}$\\
$\tt{[19]: -52z_{1} +4z_{2} +12z_{3} +12z_{4} +4z_{6}\geq 0}$ \\
$\tt{[2]: +6z_{1} -54z_{2} +10z_{3} +10z_{4} +2z_{6} +2z_{7}\geq 0}$\\
$\tt{[109]:\color{red}{+8z_{1} -8z_{3} -8z_{4} +8z_{7}\geq 0}}$\\
$\tt{[16]:-48z_{5} +16z_{6} +16z_{7}\geq 0}$\\
$\tt{[20]:+5z_{1} +1z_{2} +7z_{3} +7z_{4} +13z_{5} -41z_{6} +15z_{7}\geq 0}$\\
$\tt{[40]: +12z_{3} +12z_{4} +12z_{5} +12z_{6} -36z_{7}\geq 0}$ \\
\end{tabular}
\end{center}
\end{proof}
\begin{cor}\label{morrcount3}
Let $\mathcal{A}$ be  a UC family such that $U(\mathcal{A})=[n]$ and $\emptyset \in \mathcal{A}$. A nonempty $Z(\mathcal{A})$ does not necessarily imply that there exists a feasible solution of $I^{\mathcal{A}}$.
\end{cor}
\begin{proof}
From Proposition~\ref{morrcount1} combined with Proposition~\ref{morrcount2}, followed by Corollary~\ref{rqz}. 
\end{proof}
\begin{cor}\label{counterVaughan}
Let $\mathcal{A}$ be a UC family such that $U(\mathcal{A})=[n]$ and $\emptyset \in \mathcal{A}$. A solution to the system of equations from Question~\ref{vaughanq} in $y \in  \mathbb{R}^n_{\geq 0}$ such that $\sum_{i\in [n]}y_i \leq 1$ does not necessarily imply that $P^{\mathcal{A}}$ is nonempty.
\end{cor}
\begin{proof}
Considering $\mathcal{D}$ as above with Observation~\ref{vaugmorr} and the proof of Corollary~\ref{morrcount3} yields the desired result. Alternatively in the appendix we show that, given $\mathcal{D}$, there exists a solution to the system of equations from Question~\ref{vaughanq} in $y \in  \mathbb{R}^n_{\geq 0}$ such that $\sum_{i\in [n]}y_i \leq 1$. This coupled with Proposition~\ref{morrcount2} and Corollary~\ref{rqz}, yields the result again.
\end{proof}
\section*{Conclusion}
In this work we design a cutting-plane algorithm that determines if a given UC
family necessarily implies Frankl’s conjecture for all families that contain it.
By employing exact rational integer programming and highly redundant verifi-
cation routines, we classify more previously unknown miminal non-isomorphic
FC-families than the total output of the past twenty-five years of research on the
topic. The effects of safely automating the discovery of FC-families allow us to
answer several open questions of Morris~\cite{Morris} and Vaughan~\cite{Vaughan2}. In particular, the
counterexamples we exhibit to settle open questions of interest require no trust
from the reader, in the sense that they are independent of the complex optimization
processeses that led to them, and can be checked by hand. Furthermore, our
framework can be used to improve several other results in the following ways:
\begin{itemize}
\item Since Algorithm~\ref{row generation} determines exactly whether a given UC family $\mathcal{A}$ if FC or Non\textendash FC for $6 \leq n \leq 10$, lower bounds for previously unknown $FC(k,n)$ in this range become trivial to obtain. Furthermore when coupled with a computer algebra system or graph isomorphism software to obtain the isomorphism types of generators, upper or exact bounds for previously unknown $FC(k,n)$ are obtained in the aforementioned range.
\vspace{0.3cm}
\item  The approach of Morris~\cite{Morris} for the classification of FC-families on five elements
lends itself well to being generalized within our framework. The number
of minimal non-isomorphic generators for FC-families seems to quickly
grow for $n \geq 6$, but we believe a complete classification for $n = 6$ is possible
with routine work.
\item Proving the 3-sets conjecture of Morris~\cite{Morris}, by recovering the arguments of
Vaughan~\cite{Vaughan3} through a classification of $FC(3, n)$ for $7 \leq n \leq 9 $ and using
Morris’s lower bound on 3-sets, is within reach.
\end{itemize}
\section*{Acknowledgements}
The author would like to thank  Martin Gr\"otschel, Ralf Bornd\"orfer, Felipe Serrano and Axel Werner for fruitful discussions, and Ambros Gleixner and Stephen Maher for their
helpful suggestions about SCIP and exact SCIP.
\bibliographystyle{plain}
  \bibliography{PulajCuttingPlanes}
\appendix
\section{Appendix}
To check the claims of \emph{infeasibility} for the linear systems in this paper it is sufficient to ensure that the vector of values exhibited in square brackets before each row corresponds to the vector $y$ in the theorem below.
\begin{thrm}[Farkas' Lemma]
Let $A_1 \in \mathbb{R}^{m_1 \times n}$, $A_2 \in \mathbb{R}^{m_2 \times n}$  and $A_3 \in \mathbb{R}^{m_3 \times n}$. Also let $b_1 \in \mathbb{R}^{m_1}$, $b_2 \in \mathbb{R}^{m_2}$ and $b_3 \in \mathbb{R}^{m_3}$. Then the following system of linear equalities and inequalities in $x \in \mathbb{R}^{n}$ :
\begin{align*}
&A_1x = b_1 \\
&A_2x \leq b_2 \\
&A_3x \geq b_3\\
&x \geq 0
\end{align*}
\\
is infeasible if and only if there exist $y_1 \in \mathbb{R}^{m_1},y_2 \in \mathbb{R}^{m_2},y_3 \in \mathbb{R}^{m_3}$ such that:
\begin{align*}
&b_1^\top y_1 + b_2^\top y_2 + b_3^\top y_3 > 0 \\
&A_1^\top y_1 + A_2^\top y_2 + A_3^\top y_3 \leq 0 \\
&y_2 \leq 0\\
&y_3 \geq 0
\end{align*}
\\

\end{thrm}
\begin{proof}[Proof of Proposition \ref{morriscounter}]
We identify sets in $\mathcal{P}([6])$ with the columns in the matrix below. For each column, the number on the top row represents its corresponding variable index in $IP(\gen{\mathcal{S}'},c)$. Column $\color{red}{c}$ corresponds to a weight vector for the elements in $[n]$. The columns representing families of sets $\mathcal{S}'$ and $\mathcal{T}$ are colored red and blue, respectively. As previously, $\gen{\mathcal{S}'} = \mathcal{S}' \cup \mathcal{T}$.
\vspace{7mm}
\\
\[
\begin{tikzpicture}[ transform canvas={scale=0.69} ]
\matrix (m)[matrix of math nodes]
{   \color{red}{c} &  \color{red}{0} &  \color{red}{1} &  \color{red}{2} &  \color{red}{3} &  \color{red}{4} &  \color{red}{5} &  \color{red}{6} &  \color{red}{7} &  \color{red}{8} &  \color{red}{9} &  \color{red}{10} &  \color{red}{11} &  \color{red}{12} &  \color{red}{13} &  \color{red}{14} &  \color{red}{15} &  \color{red}{16} &  \color{red}{17} &  \color{red}{18} &  \color{red}{19} &  \color{red}{20} &  \color{red}{21} &  \color{red}{22} &  \color{red}{23} &  \color{red}{24} &  \color{red}{25} &  \color{red}{26} &  \color{red}{27} &  \color{red}{28} &  \color{red}{29} &  \color{red}{30} &  \color{red}{31} \\ 
16& \color{blue}{1}& \color{red}{1}& 1& 1& 1& 1& 1& 1&  \color{blue}{1}& 1& 1& 1& \color{red}{1}& 1& 1& 1& \color{blue}{1}& 1& 1& \color{red}{1}& 1& 1& 1& 1& 1& 1& 1& 1& 1& 1& 1& 1\\ 
8& \color{blue}{1}& \color{red}{1}& 1& 1& 1& 1& 1& 1&  \color{blue}{1}& 1& 1& 1& \color{red}{1}& 1& 1& 1& \color{blue}{0}& 0& 0& \color{red}{0}& 0& 0& 0& 0& 0& 0& 0& 0& 0& 0& 0& 0\\ 
12& \color{blue}{1}& \color{red}{1}& 1& 1& 1& 1& 1& 1&  \color{blue}{0}& 0& 0& 0& \color{red}{0}& 0& 0& 0& \color{blue}{1}& 1& 1& \color{red}{1}& 1& 1& 1& 1& 0& 0& 0& 0& 0& 0& 0& 0\\ 
20& \color{blue}{1}& \color{red}{1}& 1& 1& 0& 0& 0& 0&  \color{blue}{1}& 1& 1& 1& \color{red}{0}& 0& 0& 0& \color{blue}{1}& 1& 1& \color{red}{1}& 0& 0& 0& 0& 1& 1& 1& 1& 0& 0& 0& 0\\ 
17& \color{blue}{1}& \color{green}{1}& 0& 0& 1& 1& 0& 0&  \color{blue}{1}& 1& 0& 0& \color{red}{1}& 1& 0& 0& \color{blue}{1}& 1& 0& \color{red}{0}& 1& 1& 0& 0& 1& 1& 0& 0& 1& 1& 0& 0\\ 
15& \color{blue}{1}& 0& 1& 0& 1& 0& 1& 0&  \color{blue}{1}& 0& 1& 0&  \color{red}{1}& 0& 1& 0& \color{blue}{1}& 0& 1& \color{red}{0}& 1& 0& 1& 0& 1& 0& 1& 0& 1& 0& 1& 0\\ 
};

\end{tikzpicture}
\]
\\
\vspace{7mm}
\\
\[
\begin{tikzpicture}[ transform canvas={scale=0.65} ]
\matrix (m)[matrix of math nodes]
{   \color{red}{32} &  \color{red}{33} &  \color{red}{34} &  \color{red}{35} &  \color{red}{36} &  \color{red}{37} &  \color{red}{38} &  \color{red}{39} &  \color{red}{40} &  \color{red}{41} &  \color{red}{42} &  \color{red}{43} &  \color{red}{44} &  \color{red}{45} &  \color{red}{46} &  \color{red}{47} &  \color{red}{48} &  \color{red}{49} &  \color{red}{50} &  \color{red}{51} &  \color{red}{52} &  \color{red}{53} &  \color{red}{54} &  \color{red}{55} &  \color{red}{56} &  \color{red}{57} &  \color{red}{58} &  \color{red}{59} &  \color{red}{60} &  \color{red}{61} &  \color{red}{62} &  \color{red}{63} \\ 
0& 0& 0& 0& 0& 0& 0& 0& 0& 0& 0& 0& 0& 0& 0& 0& 0& 0& 0& 0& 0& 0& 0& 0& \color{red}{0}& 0& 0& 0& 0& 0& 0& \color{red}{0}\\ 
1& 1& 1& 1& 1& 1& 1& 1& 1& 1& 1& 1& 1& 1& 1& 1& 0& 0& 0& 0& 0& 0& 0& 0& \color{red}{0}& 0& 0& 0& 0& 0& 0& \color{red}{0}\\ 
1& 1& 1& 1& 1& 1& 1& 1& 0& 0& 0& 0& 0& 0& 0& 0& 1& 1& 1& 1& 1& 1& 1& 1& \color{red}{0}& 0& 0& 0& 0& 0& 0& \color{red}{0}\\ 
1& 1& 1& 1& 0& 0& 0& 0& 1& 1& 1& 1& 0& 0& 0& 0& 1& 1& 1& 1& 0& 0& 0& 0& \color{red}{1}& 1& 1& 1& 0& 0& 0& \color{red}{0}\\ 
1& 1& 0& 0& 1& 1& 0& 0& 1& 1& 0& 0& 1& 1& 0& 0& 1& 1& 0& 0& 1& 1& 0& 0& \color{red}{1}& 1& 0& 0& 1& 1& 0& \color{red}{0}\\ 
1& 0& 1& 0& 1& 0& 1& 0& 1& 0& 1& 0& 1& 0& 1& 0& 1& 0& 1& 0& 1& 0& 1& 0& \color{red}{1}& 0& 1& 0& 1& 0& 1& \color{red}{0}\\ 
};

\end{tikzpicture}
\]
\\
\vspace{5mm}
\\
We prove that $IP(\gen{\mathcal{S}'},c)$ with some added valid FC and FC-chain inequalities is infeasible by branching on $x_0$ and showing that the linear relaxations of the two subproblems are infeasible. We denote an explicit FC-chain by  $B_i \xrightarrow{S} B_k \xrightarrow{U} \ldots B_p\xrightarrow{T} B_j$, where $S,U,T$ satisfy either condition listed in Definition~\ref{fcchain}. When needed we specify which type of inequalities form an FC-chain by $S^{UC}, U^{UC}, T^{UC}$ for UC inequalities, and $S^{FS}, U^{FS}, T^{FS}$ for FS inequalites. We show infeasibility by explicitly exhibiting Farkas dual values (shown in square brackets) for each row of some irreducible infeasible subset of constraints. It suffices to show the infeasibilty of the following system (trivial inequalities not shown):
\begin{enumerate}
\item $\tt{[44]:x_{0}=1}$.

\item UC inequalites: 

$\tt{[-2]:x_{11} + x_{45} - x_{9} \leq 1}$, $\tt{[-3]:x_{13} + x_{59} - x_9 \leq 1}$, 

$\tt{[-2]:x_{14} + x_{43} - x_{10} \leq 1}$, $\tt{[-1]:x_{22} + x_{61}  - x_{20}\leq 1}$, 

$[-3]:\tt{x_{23} + x_{60}  - x_{20}\leq 1}$, $\tt{[-1]:x_{35} + x_{45} - x_{33}\leq 1}$, 

$\tt{[-3]:x_{35} + x_{62} - x_{34}\leq 1}$, $\tt{[-6]:x_{37} + x_{59} - x_{33}\leq 1}$, 

$\tt{[-1]:x_{38} + x_{43} - x_{34} \leq 1}$, $\tt{[-3]:x_{38} + x_{45} - x_{36} \leq 1}$,

 $\tt{[-1]:x_{38} + x_{61} - x_{36} \leq 1}$, $\tt{[-1]:x_{39} + x_{44} - x_{36} \leq 1}$, 

$\tt{[-2]:x_{42} + x_{55} - x_{34} \leq 1}$, $\tt{[-2]:x_{53} + x_{43} - x_{33} \leq 1}$, 

$\tt{[-5]:x_{54} + x_{43} - x_{34}\leq 1}$, $\tt{[-3]:x_{44} + x_{55} - x_{36} \leq 1}$, 

$\tt{[-4]:x_{47} + x_{49} - x_{33}\leq 1}$.
\item FS inequalities: 

$\tt{[0]:x_{47} - x_{1} \leq 0}$, $\tt{[-6]:x_{63} - x_{1} \leq 0}$, $\tt{[-14]:x_{63} - x_{8} \leq 0}$, 

$\tt{[-1]:x_{7} - x_{4} \leq 0}$,  $\tt{[-16]:x_{55} - x_{4} \leq 0}$,  $\tt{[-12]:x_{63} - x_{12} \leq 0}$, 

$\tt{[-3]:x_{14} - x_{2} \leq 0}$, $\tt{[-24]:x_{46} - x_{2} \leq 0}$, $\tt{[-12]:x_{47} - x_{3} \leq 0}$,  

$\tt{[-21]:x_{61} - x_{17} \leq 0}$, $\tt{[-19]:x_{62} - x_{18} \leq 0}$, $\tt{[-4]:x_{63} - x_{19} \leq 0}$, 

$\tt{[-24]:x_{31} - x_{24} \leq 0}$, $\tt{[-1]:x_{37} - x_{32} \leq 0}$, $\tt{[-4]:x_{38} - x_{32} \leq 0}$, 

$\tt{[-23]:x_{39} - x_{32} \leq 0}$, $\tt{[-16]:x_{47} - x_{40} \leq 0}$,  $\tt{[-11]:x_{55} - x_{48} \leq 0}$, 

 $\tt{[-8]:x_{63} - x_{56}\leq 0}$.
\item FC inequalities:  

$\tt{[-2]:x_{15} + x_{53} - x_{1} - x_{5} \leq 0}$, $\tt{[-7]:x_{15} + x_{57} - x_{1} - x_{9} \leq 0}$, 

$\tt{[-9]:x_{58} + x_{15} - x_{8} - x_{10} \leq 0}$, $\tt{[-2]:x_{15} + x_{59} - x_{1} - x_{11} \leq 0}$, 

$\tt{[-7]:x_{45} + x_{23} - x_{1} - x_{5} \leq 0}$, $\tt{[-5]:x_{60} + x_{23} - x_{20} - x_{16} \leq 0}$, 

$\tt{[-1]:x_{61} + x_{23} - x_{21} - x_{16} \leq 0}$, $\tt{[-5]:x_{45} + x_{27} - x_{1} - x_{9} \leq 0}$, 

$\tt{[-3]:x_{27} + x_{61} - x_{25} - x_{16} \leq 0}$, $\tt{[0]:x_{43} + x_{29} - x_{1} - x_{9} \leq 0}$, 

$\tt{[-5]:x_{54} + x_{29} - x_{20} - x_{16} \leq 0}$, $\tt{[-6]:x_{29} + x_{59} - x_{16} - x_{25} \leq 0}$, 

$\tt{[-4]:x_{43} + x_{30} - x_{10} - x_{8} \leq 0}$, $\tt{[-2]:x_{53} + x_{30} - x_{16} - x_{20} \leq 0}$, 

$\tt{[-7]:x_{59} + x_{30} - x_{16} - x_{26} \leq 0}$, $\tt{[-4]:x_{31} + x_{60} - x_{16} - x_{28}\leq 0}$, 

$\tt{[-1]:x_{43} + x_{45} - x_{8}- x_{41} \leq 0}$, $\tt{[-1]:x_{43} + x_{62} - x_{8}- x_{42}  \leq 0}$, 

$\tt{[-3]:x_{47} + x_{62} - x_{8}- x_{46} \leq 0}$, $\tt{[-3]:x_{51} + x_{62} - x_{16} - x_{50}  \leq 0}$, 

$\tt{[-7]:x_{7}+ x_{54}- x_{4}- x_{6} \leq 0}$, $\tt{[-9]:x_{51} + x_{53} - x_{48} - x_{49} \leq 0}$.
\item WV inequality: 
 
$\tt{[-0.5]:88x_{0}+58x_{1}+54 x_{2}+24 x_{3}+48 x_{4}+18x_{5}+14 x_{6}-16 x_{7}+64 x_{8}}$

$\tt{+34 x_{9}+30 x_{10}+24x_{12}- 6 x_{13} - 10 x_{14} - 40 x_{15} + 72 x_{16} + 42 x_{17}}$ 

$\tt{+ 38x_{18}+ 8x_{19} + 32x_{20} + 2x_{21} - 2x_{22}- 32x_{23} + 48x_{24}+ 18x_{25}}$

$\tt{+ 14x_{26} - 16x_{27} + 8x_{28} - 22x_{29} - 26x_{30}- 56x_{31} + 56x_{32}+ 26x_{33}}$ 
                  
$\tt{+ 22x_{34} - 8x_{35} + 16x_{36}- 14x_{37} - 18x_{38} - 48x_{39} + 32x_{40} + 2x_{41}}$ 

$\tt{- 2x_{42}- 32x_{43} - 8x_{44} - 38x_{45} - 42x_{46} - 72x_{47} + 40x_{48}}$

$\tt{+ 10x_{49} + 6x_{50} - 24x_{51} - 30x_{53}- 34x_{54}- 64x_{55} + 16x_{56} - 14x_{57}}$

$\tt{- 18 x_{58} - 48 x_{59} - 24 x_{60} - 54 x_{61}- 58 x_{62} - 88 x_{63} \leq -1}$.
\end{enumerate}
Furthermore we show that the following system of constraints is infeasible (trivial ones not shown):
\begin{enumerate}
\item $\tt{[-186.5]:x_0 = 0}$
\item FS inequalities: 

$\tt{[-7.5]: x_1 - x_0 \leq 0}$, $\tt{[-10]: x_6 - x_0 \leq 0}$, $\tt{[-8.5]: x_{11} - x_0\leq 0}$,

$\tt{[0]: x_{19} - x_0 \leq 0}$,  $\tt{[-8.5]:x_{23} - x_0 \leq 0}$,  $\tt{[-4]:x_{35} - x_0 \leq 0}$, 

$\tt{[-7]:x_{37} - x_0 \leq 0}$, $\tt{[-9]:x_{38} - x_0 \leq 0}$, $\tt{[-24]:x_{39} - x_0 \leq 0}$,  

$\tt{[-2.5]:x_{41} - x_0 \leq 0}$, $\tt{[-16]:x_{44} - x_0 \leq 0}$, $\tt{[-21]:x_{46} - x_0\leq 0}$, 

$\tt{[-15]:x_{47} - x_0 \leq 0}$, $\tt{[-6]:x_{50} - x_0\leq 0}$, $\tt{[-19]:x_{55} - x_0 \leq 0}$, 

$\tt{[-5.5]:x_{56} - x_0 \leq 0}$, $\tt{[-17]:x_{59} - x_0 \leq 0}$,  $\tt{[-16]:x_{61} - x_0 \leq 0}$,  

$\tt{[-11]:x_{62} - x_0\leq 0}$,  $\tt{[-23]:x_{63} - x_0 \leq 0}$, $\tt{[-12]:x_{13} - x_{12}  \leq 0}$,  

$\tt{[-12.5]:x_{14} - x_{2} \leq 0}$,  $\tt{[-12.5]: x_{22} - x_{18} \leq 0}$,  $\tt{[-6.5]:x_{62} - x_{18} \leq 0}$,

$\tt{[-1]:x_{42} - x_{40} \leq 0}$,  $\tt{[-7]:x_{51} - x_{48} \leq 0}$, $\tt{[-7.5]:x_{43} - x_{8}\leq 0}$, 

$\tt{[-5.5]: x_{29}  - x_{17}\leq 0}$, $\tt{[-5.5]:x_{61} - x_{9} \leq 0}$, $\tt{[0]:x_{63} - x_{56} \leq 0}$.
\item FC inequalities:  

$\tt{[-7.5]:x_{15} + x_{45} - x_{1}- x_{13}\leq 0}$, $\tt{[-9]:x_{15} + x_{53} - x_{1}- x_{5} \leq 0}$, 

$\tt{[-3.5]:x_{15} + x_{57}  - x_{1}- x_{9}\leq 0}$, $\tt{[-7.5]:x_{23} + x_{62} - x_{16} - x_{22} \leq 0}$, 

$\tt{[-8]:x_{27} + x_{45} - x_{1}- x_{9} \leq 0}$, $\tt{[-8.5]:x_{31} + x_{43} - x_{1}- x_{11}\leq 0}$, 

$\tt{[-1]: x_{31} + x_{53} - x_{16} - x_{21} \leq 0}$, $\tt{[-3.5]:x_{45} + x_{57} - x_{8}- x_{41} \leq 0}$, 

$\tt{[-9]:x_{55} + x_{58}  - x_{16} - x_{50} \leq 0}$, $\tt{[-17]:x_{7}+ x_{54} - x_{4}- x_{6} \leq 0}$, 

$\tt{[-5]:x_{51} + x_{53} - x_{48} - x_{49}\leq 0}$.
\item FC-chain inequalities (it is straightforward to check that the explicit chains, where we identify sets with their respective column numbers, work as required by Proposition~\ref{FCineq}):
	
$\tt{[-1.5]:x_{29} + x_{47} + x_{61} + x_{63} - x_{8}- x_{13} - x_{17}  - x_{56} \leq 0}$, 

$(29 \xrightarrow{19^{FS}} 17)$, $(63 \xrightarrow{8^{FS}} 8)$, $(47 \xrightarrow{8^{FS}} 8)$, $(61 \xrightarrow{56^{FS}} 56)$, $(63 \xrightarrow{56^{FS}} 56)$, 

$(47 \xrightarrow{29^{UC}} 13)$, $(61 \xrightarrow{19^{FS}} 17)$.

$\tt{[-4]:x_{29} + x_{61} + x_{62} - x_{16} - x_{17} - x_{28} \leq 0}$, 

$(29 \xrightarrow{16^{FS}} 16)$, $(61 \xrightarrow{19^{FS}} 17)$, $(62 \xrightarrow{19^{FS}} 18 \xrightarrow{16^{FS}} 16)$, 

$(29 \xrightarrow{62^{UC}} 28)$.

$\tt{[-7.5]:x_{30} + x_{31} + x_{47} + x_{63} - x_{8}- x_{14} - x_{16} - x_{24}  \leq 0}$, 

$(63 \xrightarrow{8^{FS}} 8)$, $(47 \xrightarrow{8^{FS}} 8)$, $(63 \xrightarrow{16^{FS}} 16)$, $(47 \xrightarrow{30^{UC}} 14)$, $(30 \xrightarrow{16^{FS}} 16)$, 

$(30 \xrightarrow{56^{FS}} 24)$, $(31 \xrightarrow{16^{FS}} 16)$, $(31 \xrightarrow{56^{FS}} 24)$.

$\tt{[-4]:x_{30} + x_{31} + x_{55} - x_{19} - x_{22} - x_{24}\leq 0}$, 

$(30 \xrightarrow{56^{FS}} 24)$, $(31 \xrightarrow{56^{FS}} 24)$, $(31 \xrightarrow{19^{FS}} 19)$, $(55 \xrightarrow{30^{UC}} 22)$, $(55 \xrightarrow{19^{FS}} 19)$.

$\tt{[-7]:x_{30} + x_{31} + x_{59} - x_{16} - x_{24} - x_{26}  \leq 0}$, 

$(30 \xrightarrow{56^{FS}} 24)$, $(31 \xrightarrow{56^{FS}} 24)$, $(31 \xrightarrow{16^{FS}} 16)$, $(30 \xrightarrow{16^{FS}} 16)$, 

$(59 \xrightarrow{16^{FS}} 16)$, $(59 \xrightarrow{30^{UC}} 26)$.

$\tt{[0]:x_{47} + x_{54} + x_{63} - x_{8}- x_{16} - x_{38} \leq 0}$, 

$(63 \xrightarrow{8^{FS}} 8)$, $(63 \xrightarrow{16^{FS}} 16)$, $(47 \xrightarrow{8^{FS}} 8)$, $(54 \xrightarrow{16^{FS}} 16)$, $(54 \xrightarrow{47^{UC}} 38)$.

$\tt{[-12]: x_{47} + x_{60} + x_{63} -x_{8} -x_{44} -x_{56}\leq 0}$.

$(47 \xrightarrow{8^{FS}} 8)$, $(63 \xrightarrow{8^{FS}} 8)$, $(63 \xrightarrow{56^{FS}} 56)$, $(60 \xrightarrow{56^{FS}} 56)$, $(60 \xrightarrow{47^{UC}} 44)$.

\item WV inequality: 
 
$\tt{[-0.5]:58x_{1}+54 x_{2}+24 x_{3}+48 x_{4}+18x_{5}+14 x_{6}-16 x_{7}+64 x_{8}}$

$\tt{+34 x_{9}+30 x_{10}+24x_{12}- 6 x_{13} - 10 x_{14} - 40 x_{15} + 72 x_{16} + 42 x_{17}}$ 

$\tt{+ 38x_{18}+ 8x_{19} + 32x_{20} + 2x_{21} - 2x_{22}- 32x_{23} + 48x_{24}+ 18x_{25}}$

$\tt{+ 14x_{26} - 16x_{27} + 8x_{28} - 22x_{29} - 26x_{30}- 56x_{31} + 56x_{32}+ 26x_{33}}$ 
                  
$\tt{+ 22x_{34} - 8x_{35} + 16x_{36}- 14x_{37} - 18x_{38} - 48x_{39} + 32x_{40} + 2x_{41}}$ 

$\tt{- 2x_{42}- 32x_{43} - 8x_{44} - 38x_{45} - 42x_{46} - 72x_{47} + 40x_{48}}$

$\tt{+ 10x_{49} + 6x_{50} - 24x_{51} - 30x_{53}- 34x_{54}- 64x_{55} + 16x_{56} - 14x_{57}}$

$\tt{- 18 x_{58} - 48 x_{59} - 24 x_{60} - 54 x_{61}- 58 x_{62} - 88 x_{63} \leq -1}$.
\end{enumerate}
\end{proof}
\begin{proof}[Another proof of Corollary~\ref{counterVaughan}]
Next, we explicitly answer Vaughan's question in the negative. Given $\gen{\mathcal{S}}$, we show that there exists a nonnegative solution to the system of equations in Question~\ref{vaughanq} such that $\sum_{i\in [n]}y_i \leq 1$, where $\mathcal{S}$ is defined as in the counterexample to Morris's question. Furthermore the order of display of equations is the same as previously.
\begin{center}
\begin{tabular}{@{}l@{}}
$\tt{24y_{1} +80y_{2} +88y_{3} +88y_{4} +76y_{5} +80y_{6} +76y_{7}= 76}$ \\
$\tt{80y_{1} +20y_{2} +84y_{3} +84y_{4} +74y_{5} +76y_{6} +76y_{7}= 74}$ \\
$\tt{92y_{1} +88y_{2} +44y_{3} +108y_{4} +86y_{5} +88y_{6} +96y_{7}= 86}$ \\
$\tt{92y_{1} +88y_{2} +108y_{3} +44y_{4} +86y_{5} +88y_{6} +96y_{7}= 86}$ \\
$\tt{80y_{1} +80y_{2} +80y_{3} +80y_{4} +32y_{5} +96y_{6} +96y_{7}= 80}$ \\
$\tt{92y_{1} +88y_{2} +94y_{3} +94y_{4} +100y_{5} +46y_{6} + 102y_{7}= 87}$ \\
$\tt{92y_{1} +92y_{2} +104y_{3} +104y_{4} + 104y_{5} +104y_{6} +56y_{7}= 92}$ \\

\end{tabular}
\end{center}
Let $\bar{y}_1 = \frac{28304}{309701}$,$\bar{y}_2 = \frac{60251}{738922}$, $\bar{y}_3 = \frac{94175}{606582}$, $\bar{y}_4 = \frac{94175}{606582}$, $\bar{y}_5 = \frac{63417}{493048}$, $\bar{y}_6 = \frac{158373}{872233}$, $\bar{y}_7 = \frac{95228}{462227}$. Then $\bar{y} \in \mathbb{Q}^7_{\geq 0}$ such that $\bar{y}=(\bar{y}_1,\bar{y}_2, \ldots, \bar{y}_7)$ is a solution to the system of linear equations above such that the following holds,
\begin{equation}
\sum_{i \in [7]}\bar{y}_i= \frac{6896010572642828356716603827169373}{6898390222382701705240892810504568} < 1 \nonumber.
\end{equation}
\end{proof}

 %Let $n_0$ be the smallest positive integer such that $m \leq \dbinom{n_0}{k}$. Let $s_0 := \dbinom{n_0}{k}$, $t_0:=\dbinom{s_0}{m}$, $h_0:= s_0t_0$. Let $n_{k+1}=n_k +1$ for all $0\leq k$ such that $n_0 \leq n_k \leq n-1$. As a result, $s_n,t_n,h_n$ are also defined. Finally, let $c_0:=h_0$. 
%\begin{proposition}
%Let $c_n$ be the number of $m$-$k$-sets on $n$ elements that cover the ground set. Then $c_n$ is determined by the following recursion:
%\begin{equation}
%c_n = h_n - \displaystyle\sum_{i=0}^{n-1} \dbinom{n}{n_i}c_i \nonumber.
%\end{equation}
%\end{proposition}
%\begin{proof}
%To see that the recursion holds, it is sufficient to consider all m-k-sets on n elements which do not cover the ground set and the relation easily follows.
%\end{proof} 
\begin{table}[h!]
\begin{center}
\scalebox{0.85}{
\begin{tabular}{|l|l|}
\hline
\multicolumn{2}{|c|}{Previously unknown minimal nonisomorphic generators for FC-families on $[6]$}\\\hline
\hline
%456, 356, 126 & $1 \mapsto 5$, $2 \mapsto 5$, $3 \mapsto 5$, $4 \mapsto 5$, $5 \mapsto 8$, $6 \mapsto 10$  \\\hline
1256, 3456, 456, 236 & $1 \mapsto 7$, $2 \mapsto 15$, $3 \mapsto 15$, $4 \mapsto 11$, $5 \mapsto 14$, $6 \mapsto 20$  \\\hline
12456, 2346, 456, 356 & $1 \mapsto 1$, $2 \mapsto 3$, $3 \mapsto 5$, $4 \mapsto 5$, $5 \mapsto 6$, $6 \mapsto 7$  \\\hline
12345, 1356, 456, 356 & $1 \mapsto 1$, $2 \mapsto 2$, $3 \mapsto 4$, $4 \mapsto 4$, $5 \mapsto 5$, $6 \mapsto 5$  \\\hline
12345, 2346, 456, 236 & $1 \mapsto 2$, $2 \mapsto 3$, $3 \mapsto 4$, $4 \mapsto 3$, $5 \mapsto 4$, $6 \mapsto 5$  \\\hline
12345, 2346, 456, 236 & $1 \mapsto 2$, $2 \mapsto 3$, $3 \mapsto 4$, $4 \mapsto 3$, $5 \mapsto 4$, $6 \mapsto 5$  \\\hline
12346, 1256, 456, 356 & $1 \mapsto 4$, $2 \mapsto 4$, $3 \mapsto 7$, $4 \mapsto 7$, $5 \mapsto 9$, $6 \mapsto 10$  \\\hline
12356, 1345, 456, 236 & $1 \mapsto 8$, $2 \mapsto 12$, $3 \mapsto 16$, $4 \mapsto 15$, $5 \mapsto 17$, $6 \mapsto 20$  \\\hline
12356, 1234, 456, 356 & $1 \mapsto 8$, $2 \mapsto 8$, $3 \mapsto 24$, $4 \mapsto 24$, $5 \mapsto 27$, $6 \mapsto 29$  \\\hline
12456, 1356, 456, 326 & $1 \mapsto 45$, $2 \mapsto 71$, $3 \mapsto 77$, $4 \mapsto 59$, $5 \mapsto 74$, $6 \mapsto 103$  \\\hline
136, 2456, 3456, 456, 123 & $1 \mapsto 6$, $2 \mapsto 5$, $3 \mapsto 7$, $4 \mapsto 3$, $5 \mapsto 3$, $6 \mapsto 6$  \\\hline
136, 1256, 3456, 456, 123 & $1 \mapsto 2$, $2 \mapsto 1$, $3 \mapsto 2$, $4 \mapsto 1$, $5 \mapsto 1$, $6 \mapsto 2$  \\\hline
2346, 3456, 2456, 2356, 1234 & $1 \mapsto 2$, $2 \mapsto 5$, $3 \mapsto 5$, $4 \mapsto 5$, $5 \mapsto 4$, $6 \mapsto 5$  \\\hline
3456, 2456, 2356, 1346, 1246, 1234 & $1 \mapsto 3$, $2 \mapsto 4$, $3 \mapsto 4$, $4 \mapsto 4$, $5 \mapsto 3$, $6 \mapsto 4$  \\\hline
3456, 2456, 2356, 1346, 1245, 1234 & $1 \mapsto 1$, $2 \mapsto 2$, $3 \mapsto 2$, $4 \mapsto 2$, $5 \mapsto 2$, $6 \mapsto 2$  \\\hline
3456, 2456, 1456, 1236, 1235, 1234 & $1 \mapsto 1$, $2 \mapsto 1$, $3 \mapsto 1$, $4 \mapsto 1$, $5 \mapsto 1$, $6 \mapsto 1$  \\\hline
3456, 2456, 1356, 1246, 1235, 1234 & $1 \mapsto 1$, $2 \mapsto 1$, $3 \mapsto 1$, $4 \mapsto 1$, $5 \mapsto 1$, $6 \mapsto 1$  \\\hline
3456, 2456, 2356, 2346, 1456, 1356 & $1 \mapsto 8$, $2 \mapsto 14$, $3 \mapsto 15$, $4 \mapsto 15$, $5 \mapsto 16$, $6 \mapsto 19$  \\\hline
3456, 2456, 2356, 2346, 1456, 1236 & $1 \mapsto 3$, $2 \mapsto 4$, $3 \mapsto 4$, $4 \mapsto 4$, $5 \mapsto 4$, $6 \mapsto 5$  \\\hline
3456, 2456, 2356, 1456, 1356, 1234 & $1 \mapsto 1$, $2 \mapsto 1$, $3 \mapsto 1$, $4 \mapsto 1$, $5 \mapsto 1$, $6 \mapsto 1$  \\\hline
3456, 2456, 2356, 1456, 1346, 1245 & $1 \mapsto 2$, $2 \mapsto 2$, $3 \mapsto 2$, $4 \mapsto 3$, $5 \mapsto 3$, $6 \mapsto 3$  \\\hline
3456, 2456, 2356, 1456, 1346, 1235 & $1 \mapsto 5$, $2 \mapsto 4$, $3 \mapsto 5$, $4 \mapsto 5$, $5 \mapsto 6$, $6 \mapsto 6$  \\\hline
3456, 2456, 2356, 1456, 1236, 1235 & $1 \mapsto 2$, $2 \mapsto 3$, $3 \mapsto 3$, $4 \mapsto 2$, $5 \mapsto 3$, $6 \mapsto 3$  \\\hline
3456, 2456, 2356, 1456, 1236, 1234 & $1 \mapsto 4$, $2 \mapsto 5$, $3 \mapsto 5$, $4 \mapsto 5$, $5 \mapsto 4$, $6 \mapsto 5$  \\\hline
3456, 2456, 2356, 1346, 1345, 1246 & $1 \mapsto 2$, $2 \mapsto 2$, $3 \mapsto 3$, $4 \mapsto 3$, $5 \mapsto 3$, $6 \mapsto 3$  \\\hline
3456, 2456, 2356, 1346, 1246, 1235 & $1 \mapsto 3$, $2 \mapsto 4$, $3 \mapsto 4$, $4 \mapsto 3$, $5 \mapsto 4$, $6 \mapsto 4$  \\\hline
12346, 3456, 2456, 2356, 1456, 1356, 1256 & $1 \mapsto 5$, $2 \mapsto 5$, $3 \mapsto 5$, $4 \mapsto 5$, $5 \mapsto 6$, $6 \mapsto 7$  \\\hline
1236, 3456, 2456, 2356, 1456, 1356, 1246 & $2 \mapsto 2$, $2 \mapsto 3$, $3 \mapsto 3$, $4 \mapsto 3$, $5 \mapsto 3$, $6 \mapsto 4$  \\\hline
1456, 3456, 2456, 2356, 1346, 1246, 1236 & $2 \mapsto 3$, $2 \mapsto 2$, $3 \mapsto 3$, $4 \mapsto 3$, $5 \mapsto 3$, $6 \mapsto 4$  \\\hline
1256, 3456, 2456, 2356, 1456, 1346, 1236 & $2 \mapsto 2$, $2 \mapsto 3$, $3 \mapsto 3$, $4 \mapsto 3$, $5 \mapsto 3$, $6 \mapsto 4$  \\\hline
2356, 2456, 345, 13456, 12346 & $2 \mapsto 3$, $2 \mapsto 8$, $3 \mapsto 12$, $4 \mapsto 12$, $5 \mapsto 13$, $6 \mapsto 9$  \\\hline
1234, 1256, 246, 23456, 13456 & $2 \mapsto 9$, $2 \mapsto 12$, $3 \mapsto 7$, $4 \mapsto 11$, $5 \mapsto 7$, $6 \mapsto 11$  \\\hline
1236, 2456, 125, 23456, 13456 & $2 \mapsto 32$, $2 \mapsto 34$, $3 \mapsto 19$, $4 \mapsto 16$, $5 \mapsto 32$, $6 \mapsto 25$  \\\hline
1246, 1256, 123, 23456, 13456 & $2 \mapsto 7$, $2 \mapsto 7$, $3 \mapsto 6$, $4 \mapsto 3$, $5 \mapsto 4$, $6 \mapsto 5$  \\\hline

\end{tabular}
}
\caption{Frankl's conjecture holds for all UC families which contain the following subfamilies}
\end{center}
\end{table}

\begin{table}[h!]
\begin{center}
\scalebox{0.85}{
\begin{tabular}{|l|l|}
\hline
\multicolumn{2}{|c|}{Previously unknown minimal nonisomorphic generators for FC-families on $[7]$}\\\hline
\hline
3457, 567, 467, 123 & $1 \mapsto 1$, $2 \mapsto 1$, $3 \mapsto 3$, $4 \mapsto 5$, $5 \mapsto 5$, $6 \mapsto 6$, $7 \mapsto 7$  \\\hline
2467, 567, 347, 126 & $1 \mapsto 1$, $2 \mapsto 2$, $3 \mapsto 1$, $4 \mapsto 2$, $5 \mapsto 2$, $6 \mapsto 3$, $7 \mapsto 3$  \\\hline
357, 367, 4567, 1237 & $1 \mapsto 1$, $2 \mapsto 1$, $3 \mapsto 6$, $4 \mapsto 2$, $5 \mapsto 5$, $6 \mapsto 5$, $7 \mapsto 7$  \\\hline
356, 367, 4567, 1237 & $1 \mapsto 1$, $2 \mapsto 1$, $3 \mapsto 4$, $4 \mapsto 1$, $5 \mapsto 3$, $6 \mapsto 4$, $7 \mapsto 3$  \\\hline
257, 367, 4567, 1237 & $1 \mapsto 1$, $2 \mapsto 2$, $3 \mapsto 2$, $4 \mapsto 1$, $5 \mapsto 2$, $6 \mapsto 2$, $7 \mapsto 3$  \\\hline
256, 367, 4567, 1237 & $1 \mapsto 1$, $2 \mapsto 3$, $3 \mapsto 2$, $4 \mapsto 1$, $5 \mapsto 3$, $6 \mapsto 4$, $7 \mapsto 3$  \\\hline
346, 367, 4567, 1237 & $1 \mapsto 1$, $2 \mapsto 1$, $3 \mapsto 4$, $4 \mapsto 3$, $5 \mapsto 3$, $6 \mapsto 4$, $7 \mapsto 2$  \\\hline
245, 367, 4567, 1237 & $1 \mapsto 2$, $2 \mapsto 6$, $3 \mapsto 6$, $4 \mapsto 2$, $5 \mapsto 7$, $6 \mapsto 7$, $7 \mapsto 4$  \\\hline
246, 367, 4567, 1237 & $1 \mapsto 1$, $2 \mapsto 3$, $3 \mapsto 3$, $4 \mapsto 3$, $5 \mapsto 3$, $6 \mapsto 4$, $7 \mapsto 1$  \\\hline
235, 367, 4567, 1237 & $1 \mapsto 1$, $2 \mapsto 3$, $3 \mapsto 4$, $4 \mapsto 1$, $5 \mapsto 3$, $6 \mapsto 4$, $7 \mapsto 1$  \\\hline
234, 367, 4567, 1237 & $1 \mapsto 2$, $2 \mapsto 4$, $3 \mapsto 7$, $4 \mapsto 5$, $5 \mapsto 5$, $6 \mapsto 5$, $7 \mapsto 4$  \\\hline
12456, 34567, 267, 127 & \scriptsize{$1 \mapsto 78$, $2 \mapsto 105$, $3 \mapsto 16$, $4 \mapsto 27$, $5 \mapsto 27$, $6 \mapsto 84$, $7 \mapsto 103$} \\\hline
12456, 34567, 267, 257 & $1 \mapsto 1$, $2 \mapsto 9$, $3 \mapsto 1$, $4 \mapsto 2$, $5 \mapsto 7$, $6 \mapsto 7$, $7 \mapsto 9$  \\\hline
3467, 4567, 2367, 2345, 1357 & \scriptsize{$1 \mapsto 20$, $2 \mapsto 36$, $3 \mapsto 52$, $4 \mapsto 45$, $5 \mapsto 46$, $6 \mapsto 39$, $7 \mapsto 49$}  \\\hline
3456, 4567, 2367, 1357, 1247 & \scriptsize{$1 \mapsto 15$, $2 \mapsto 15$, $3 \mapsto 20$, $4 \mapsto 18$, $5 \mapsto 19$, $6 \mapsto 19$, $7 \mapsto 23$}  \\\hline
3456, 4567, 2367, 1357, 1246 & \scriptsize{$1 \mapsto 17$, $2 \mapsto 16$, $3 \mapsto 22$, $4 \mapsto 19$, $5 \mapsto 21$, $6 \mapsto 24$, $7 \mapsto 22$}  \\\hline
2347, 4567, 3567, 1267, 1245 & \scriptsize{$1 \mapsto 69$, $2 \mapsto 91$, $3 \mapsto 71$, $4 \mapsto 93$, $5 \mapsto 87$, $6 \mapsto 81$, $7 \mapsto 103$}  \\\hline
2346, 4567, 3567, 2347, 1267 & \scriptsize{$1 \mapsto 6$, $2 \mapsto 13$, $3 \mapsto 14$, $4 \mapsto 14$, $5 \mapsto 9$, $6 \mapsto 16$, $7 \mapsto 16$} \\\hline
2345, 4567, 3567, 1247, 1236 & \scriptsize{$1 \mapsto 24$, $2 \mapsto 32$, $3 \mapsto 33$, $4 \mapsto 33$, $5 \mapsto 32$, $6 \mapsto 31$, $7 \mapsto 31$} \\\hline
2345, 4567, 2367, 1357, 1247 & $1 \mapsto 3$, $2 \mapsto 4$, $3 \mapsto 4$, $4 \mapsto 4$, $5 \mapsto 4$, $6 \mapsto 3$, $7 \mapsto 4$ \\\hline
2345, 4567, 2367, 1357, 1246 & $1 \mapsto 1$, $2 \mapsto 2$, $3 \mapsto 2$, $4 \mapsto 2$, $5 \mapsto 2$, $6 \mapsto 2$, $7 \mapsto 2$ \\\hline
1356, 4567, 2367, 2345, 1357 & $1 \mapsto 5$, $2 \mapsto 6$, $3 \mapsto 9$, $4 \mapsto 6$, $5 \mapsto 9$, $6 \mapsto 8$, $7 \mapsto 8$ \\\hline
\pbox{21cm}{12456, 13457, 23457, 12367, 12467, \\ 13467, 23467, 12567, 13567, 23567, \\ 14567, 24567, 34567} & \scriptsize{ $ 1  \mapsto 11$, $ 2  \mapsto 12$, $ 3  \mapsto 11$, $ 4  \mapsto 13$, $ 5  \mapsto 13$, $ 6  \mapsto 14$, $ 7  \mapsto 15$} \\\hline 
\pbox{21cm}{12345, 13457, 23457, 12367, 12467, \\ 13467, 23467, 12567, 13567, 23567, \\ 14567, 24567, 34567} & \scriptsize{ $ 1  \mapsto 12$, $ 2  \mapsto 12$, $ 3  \mapsto 13$, $ 4  \mapsto 13$, $ 5  \mapsto 13$, $ 6  \mapsto 12$, $ 7  \mapsto 15$} \\\hline 
\pbox{21cm}{12345, 23456, 23457, 12367, 12467, \\ 13467, 23467, 12567, 13567, 23567, \\ 14567, 24567, 34567} & $ 1  \mapsto 6$, $ 2  \mapsto 7$, $ 3  \mapsto 7$, $ 4  \mapsto 7$, $ 5  \mapsto 7$, $ 6  \mapsto 8$, $ 7  \mapsto 8$ \\\hline 
\pbox{21cm}{12345, 13456, 23457, 12367, 12467, \\ 13467, 23467, 12567, 13567, 23567, \\ 14567, 24567, 34567} & $ 1  \mapsto 6$, $ 2  \mapsto 6$, $ 3  \mapsto 6$, $ 4  \mapsto 6$, $ 5  \mapsto 6$, $ 6  \mapsto 7$, $ 7  \mapsto 7$ \\\hline 
\pbox{21cm}{23456, 12457, 13457, 23457, 12467, \\ 13467, 23467, 12567, 13567, 23567, \\ 14567, 24567, 34567} & $ 1  \mapsto 6$, $ 2  \mapsto 7$, $ 3  \mapsto 7$, $ 4  \mapsto 8$, $ 5  \mapsto 8$, $ 6  \mapsto 8$, $ 7  \mapsto 8$ \\\hline 
\pbox{21cm}{12356, 12457, 13457, 23457, 12467, \\ 13467, 23467, 12567, 13567, 23567, \\ 14567, 24567, 34567} & $ 1  \mapsto 8$, $ 2  \mapsto 8$, $ 3  \mapsto 8$, $ 4  \mapsto 8$, $ 5  \mapsto 9$, $ 6  \mapsto 9$, $ 7  \mapsto 10$ \\\hline
\pbox{21cm}{23456, 14567, 13567, 13467, 13457, \\ 13456, 12567, 12467, 12457, 12456, \\ 12367, 12357, 12356, 12347} & \scriptsize{$1 \mapsto 14$, $2 \mapsto 12$, $3 \mapsto 12$, $4 \mapsto 11$, $5 \mapsto 12$, $6 \mapsto 12$, $7 \mapsto 11$} \\\hline
\pbox{21cm}{23456, 14567, 13567, 13467, 13457, \\ 13456, 12567, 12467, 12457, 12456, \\ 12367, 12357, 12346, 12345} & $1 \mapsto 7$, $2 \mapsto 6$, $3 \mapsto 6$, $4 \mapsto 6$, $5 \mapsto 6$, $6 \mapsto 6$, $7 \mapsto 5$ \\\hline

\end{tabular}
}
\caption{Frankl's conjecture holds for all UC families which contain the following subfamilies}
\end{center}
\end{table}

\begin{table}[h!]
\begin{center}
\scalebox{0.85}{
\begin{tabular}{|l|l|}
\hline
\multicolumn{2}{|c|}{Previously unknown minimal nonisomorphic generators for FC-families on $[7]$}\\\hline
\hline
\pbox{20cm}{23456, 12357, 13457, 23457, 12467, \\ 13467, 23467, 12567, 13567, 23567, \\ 14567, 24567, 34567} & \scriptsize{ $ 1  \mapsto 7$, $ 2  \mapsto 10$, $ 3  \mapsto 10$, $ 4  \mapsto 10$, $ 5  \mapsto 11$, $ 6  \mapsto 11$, $ 7  \mapsto 12$} \\\hline
 \pbox{20cm}{12456, 12357, 13457, 23457, 12467, \\ 13467, 23467, 12567, 13567, 23567, \\ 14567, 24567, 34567} & \scriptsize{ $ 1  \mapsto 9$, $ 2  \mapsto 9$, $ 3  \mapsto 8$, $ 4  \mapsto 9$, $ 5  \mapsto 10$, $ 6  \mapsto 10$, $ 7  \mapsto 11$} \\\hline 
\pbox{20cm}{12346, 12357, 13457, 23457, 12467, \\ 13467, 23467, 12567, 13567, 23567, \\ 14567, 24567, 34567} & \scriptsize{ $ 1  \mapsto 12$, $ 2  \mapsto 12$, $ 3  \mapsto 13$, $ 4  \mapsto 13$, $ 5  \mapsto 12$, $ 6  \mapsto 13$, $ 7  \mapsto 15$} \\\hline 
\pbox{20cm}{13456, 23456, 13457, 23457, 12467, \\ 13467, 23467, 12567, 13567, 23567, \\ 14567, 24567, 34567} & $ 1  \mapsto 6$, $ 2  \mapsto 6$, $ 3  \mapsto 6$, $ 4  \mapsto 7$, $ 5  \mapsto 7$, $ 6  \mapsto 7$, $ 7  \mapsto 7$ \\\hline
 
\pbox{20cm}{12456, 23456, 13457, 23457, 12467, \\ 13467, 23467, 12567, 13567, 23567, \\ 14567, 24567, 34567} & $ 1  \mapsto 6$, $ 2  \mapsto 6$, $ 3  \mapsto 6$, $ 4  \mapsto 7$, $ 5  \mapsto 7$, $ 6  \mapsto 7$, $ 7  \mapsto 7$ \\\hline 
\pbox{20cm}{12356, 23456, 13457, 23457, 12467, \\ 13467, 23467, 12567, 13567, 23567, \\ 14567, 24567, 34567} & $ 1  \mapsto 7$, $ 2  \mapsto 7$, $ 3  \mapsto 8$, $ 4  \mapsto 8$, $ 5  \mapsto 8$, $ 6  \mapsto 9$, $ 7  \mapsto 9$ \\\hline 
\pbox{20cm}{12345, 23456, 13457, 23457, 12467, \\ 13467, 23467, 12567, 13567, 23567, \\ 14567, 24567, 34567} & $ 1  \mapsto 7$, $ 2  \mapsto 8$, $ 3  \mapsto 8$, $ 4  \mapsto 9$, $ 5  \mapsto 9$, $ 6  \mapsto 9$, $ 7  \mapsto 9$ \\\hline 
\pbox{20cm}{12356, 12456, 13457, 23457, 12467, \\ 13467, 23467, 12567, 13567, 23567, \\ 14567, 24567, 34567} & $ 1  \mapsto 6$, $ 2  \mapsto 6$, $ 3  \mapsto 6$, $ 4  \mapsto 6$, $ 5  \mapsto 6$, $ 6  \mapsto 7$, $ 7  \mapsto 7$ \\\hline 
\pbox{20cm}{12345, 12456, 13457, 23457, 12467, \\ 13467, 23467, 12567, 13567, 23567, \\ 14567, 24567, 34567} & $ 1  \mapsto 6$, $ 2  \mapsto 6$, $ 3  \mapsto 6$, $ 4  \mapsto 7$, $ 5  \mapsto 7$, $ 6  \mapsto 7$, $ 7  \mapsto 7$ \\\hline 
\pbox{20cm}{12346, 12356, 13457, 23457, 12467, \\ 13467, 23467, 12567, 13567, 23567, \\ 14567, 24567, 34567} & $ 1  \mapsto 6$, $ 2  \mapsto 6$, $ 3  \mapsto 6$, $ 4  \mapsto 6$, $ 5  \mapsto 6$, $ 6  \mapsto 7$, $ 7  \mapsto 7$ \\\hline 
\pbox{20cm}{12345, 12356, 13457, 23457, 12467, \\ 13467, 23467, 12567, 13567, 23567, \\ 14567, 24567, 34567} & \scriptsize{ $ 1  \mapsto 9$, $ 2  \mapsto 9$, $ 3  \mapsto 9$, $ 4  \mapsto 9$, $ 5  \mapsto 10$, $ 6  \mapsto 10$, $ 7  \mapsto 10$} \\\hline 
\pbox{20cm}{23456, 12347, 12357, 23457, 12467, \\ 13467, 23467, 12567, 13567, 23567, \\ 14567, 24567, 34567} & $ 1  \mapsto 5$, $ 2  \mapsto 7$, $ 3  \mapsto 7$, $ 4  \mapsto 7$, $ 5  \mapsto 7$, $ 6  \mapsto 7$, $ 7  \mapsto 8$ \\\hline 
\pbox{20cm}{13456, 12347, 12357, 23457, 12467, \\ 13467, 23467, 12567, 13567, 23567, \\ 14567, 24567, 34567} & $ 1  \mapsto 6$, $ 2  \mapsto 6$, $ 3  \mapsto 7$, $ 4  \mapsto 7$, $ 5  \mapsto 7$, $ 6  \mapsto 7$, $ 7  \mapsto 8$ \\\hline 
\pbox{20cm}{12356, 12347, 12357, 23457, 12467, \\ 13467, 23467, 12567, 13567, 23567, \\ 14567, 24567, 34567} & $ 1  \mapsto 6$, $ 2  \mapsto 7$, $ 3  \mapsto 7$, $ 4  \mapsto 6$, $ 5  \mapsto 7$, $ 6  \mapsto 7$, $ 7  \mapsto 8$ \\\hline 
\pbox{20cm}{12345, 12347, 12357, 23457, 12467, \\ 13467, 23467, 12567, 13567, 23567, \\ 14567, 24567, 34567} & \scriptsize{ $ 1  \mapsto 11$, $ 2  \mapsto 12$, $ 3  \mapsto 12$, $ 4  \mapsto 12$, $ 5  \mapsto 12$, $ 6  \mapsto 11$, $ 7  \mapsto 14$} \\\hline
\end{tabular}
}
\caption{Frankl's conjecture holds for all UC families which contain the following subfamilies}
\end{center}
\end{table}

\begin{table}[h!]
\begin{center}
\scalebox{0.85}{
\begin{tabular}{|l|l|}
\hline
\multicolumn{2}{|c|}{Previously unknown minimal nonisomorphic generators for FC-families on $[7]$}\\\hline
\hline
 
\pbox{20cm}{34567, 24567, 23567, 23467, 23457, \\ 23456, 14567, 13567, 13467, 13456, \\ 12457, 12367, 12347} & $ 1  \mapsto 7$, $ 2  \mapsto 8$, $ 3  \mapsto 9$, $ 4  \mapsto 9$, $ 5  \mapsto 8$, $ 6  \mapsto 9$, $ 7  \mapsto 9$ \\\hline 
\pbox{20cm}{34567, 24567, 23567, 23467, 23457, \\ 23456, 14567, 13567, 13467, 13456, \\ 12457, 12367, 12346} & $ 1  \mapsto 7$, $ 2  \mapsto 8$, $ 3  \mapsto 9$, $ 4  \mapsto 9$, $ 5  \mapsto 8$, $ 6  \mapsto 9$, $ 7  \mapsto 9$ \\\hline 
\pbox{20cm}{34567, 24567, 23567, 23467, 23457, \\ 23456, 14567, 13567, 13467, 13456, \\ 12457, 12367, 12345} & $ 1  \mapsto 3$, $ 2  \mapsto 3$, $ 3  \mapsto 4$, $ 4  \mapsto 4$, $ 5  \mapsto 4$, $ 6  \mapsto 4$, $ 7  \mapsto 4$ \\\hline 
\pbox{20cm}{34567, 24567, 23567, 23467, 23457, \\ 23456, 14567, 13567, 13467, 13456, \\ 12457, 12357, 12347} & $ 1  \mapsto 7$, $ 2  \mapsto 8$, $ 3  \mapsto 9$, $ 4  \mapsto 9$, $ 5  \mapsto 9$, $ 6  \mapsto 8$, $ 7  \mapsto 9$ \\\hline 
\pbox{20cm}{34567, 24567, 23567, 23467, 23457, \\ 23456, 14567, 13567, 13467, 13456, \\ 12457, 12357, 12346} & $ 1  \mapsto 3$, $ 2  \mapsto 3$, $ 3  \mapsto 4$, $ 4  \mapsto 4$, $ 5  \mapsto 4$, $ 6  \mapsto 4$, $ 7  \mapsto 4$ \\\hline 
\pbox{20cm}{34567, 24567, 23567, 23467, 23457, \\ 23456, 14567, 13567, 13467, 13456, \\ 12457, 12357, 12345} & $ 1  \mapsto 7$, $ 2  \mapsto 8$, $ 3  \mapsto 9$, $ 4  \mapsto 9$, $ 5  \mapsto 9$, $ 6  \mapsto 8$, $ 7  \mapsto 9$ \\\hline 
\pbox{20cm}{34567, 24567, 23567, 23467, 23457, \\ 23456, 14567, 13567, 13467, 13456, \\ 12457, 12347, 12346} & $ 1  \mapsto 7$, $ 2  \mapsto 8$, $ 3  \mapsto 9$, $ 4  \mapsto 9$, $ 5  \mapsto 8$, $ 6  \mapsto 9$, $ 7  \mapsto 9$ \\\hline 
\pbox{20cm}{34567, 24567, 23567, 23467, 23457, \\ 23456, 14567, 13567, 13467, 13456, \\ 12457, 12347, 12345} & $ 1  \mapsto 7$, $ 2  \mapsto 8$, $ 3  \mapsto 9$, $ 4  \mapsto 9$, $ 5  \mapsto 9$, $ 6  \mapsto 8$, $ 7  \mapsto 9$ \\\hline 
\pbox{20cm}{34567, 24567, 23567, 23467, 23457, \\ 23456, 14567, 13567, 13467, 13456, \\ 12457, 12346, 12345} & $ 1  \mapsto 7$, $ 2  \mapsto 8$, $ 3  \mapsto 9$, $ 4  \mapsto 9$, $ 5  \mapsto 9$, $ 6  \mapsto 9$, $ 7  \mapsto 8$ \\\hline 
\pbox{20cm}{34567, 24567, 23567, 23467, 23457, \\ 23456, 14567, 13567, 13467, 13456, \\ 12347, 12346, 12345} & $ 1  \mapsto 7$, $ 2  \mapsto 7$, $ 3  \mapsto 9$, $ 4  \mapsto 9$, $ 5  \mapsto 8$, $ 6  \mapsto 8$, $ 7  \mapsto 8$ \\\hline 
\pbox{20cm}{34567, 24567, 23567, 23467, 23457, \\ 23456, 14567, 13567, 13467, 12457, \\ 12456, 12347, 12346} & $ 1  \mapsto 8$, $ 2  \mapsto 9$, $ 3  \mapsto 9$, $ 4  \mapsto 10$, $ 5  \mapsto 9$, $ 6  \mapsto 9$, $ 7  \mapsto 9$ \\\hline 
\pbox{20cm}{34567, 24567, 23567, 23467, 23457, \\ 23456, 14567, 13567, 13467, 12457, \\ 12456, 12347, 12345} & $ 1  \mapsto 8$, $ 2  \mapsto 9$, $ 3  \mapsto 9$, $ 4  \mapsto 10$, $ 5  \mapsto 9$, $ 6  \mapsto 9$, $ 7  \mapsto 9$ \\\hline 
\pbox{20cm}{34567, 24567, 23567, 23467, 23457, \\ 23456, 14567, 13567, 13456, 12567, \\ 12456, 12367, 12357} & $ 1  \mapsto 7$, $ 2  \mapsto 8$, $ 3  \mapsto 8$, $ 4  \mapsto 7$, $ 5  \mapsto 9$, $ 6  \mapsto 9$, $ 7  \mapsto 8$ \\\hline 
\pbox{20cm}{34567, 24567, 23567, 23467, 23457, \\ 23456, 14567, 13567, 13456, 12567, \\ 12456, 12367, 12356} & $ 1  \mapsto 7$, $ 2  \mapsto 9$, $ 3  \mapsto 9$, $ 4  \mapsto 8$, $ 5  \mapsto 10$, $ 6  \mapsto 10$, $ 7  \mapsto 9$ \\\hline 
\pbox{20cm}{34567, 24567, 23567, 23467, 23457, \\ 23456, 14567, 13567, 13456, 12567, \\ 12456, 12356, 12347} & $ 1  \mapsto 6$, $ 2  \mapsto 7$, $ 3  \mapsto 7$, $ 4  \mapsto 7$, $ 5  \mapsto 8$, $ 6  \mapsto 8$, $ 7  \mapsto 7$ \\\hline
\end{tabular}
}
\caption{Frankl's conjecture holds for all UC families which contain the following subfamilies}
\end{center}
\end{table}

\begin{table}[h!]
\begin{center}
\scalebox{0.85}{
\begin{tabular}{|l|l|}
\hline
\multicolumn{2}{|c|}{Previously-unknown minimal nonisomorphic generators for FC-families on $[8]$}\\\hline
\hline
\pbox{20cm}{678, 578, 346, 125} & $ 1  \mapsto 1$, $ 2  \mapsto 1$, $ 3  \mapsto 1$, $ 4  \mapsto 1$, $ 5  \mapsto 2$, $ 6  \mapsto 2$, $ 7  \mapsto 2$, $ 8  \mapsto 2$ \\\hline
\pbox{20cm}{678, 458, 237, 135} & $ 1  \mapsto 1$, $ 2  \mapsto 1$, $ 3  \mapsto 2$, $ 4  \mapsto 1$, $ 5  \mapsto 2$, $ 6  \mapsto 1$, $ 7  \mapsto 2$, $ 8  \mapsto 2$ \\\hline 
\pbox{20cm}{1578, 678, 458, 237} & $ 1  \mapsto 1$, $ 2  \mapsto 3$, $ 3  \mapsto 3$, $ 4  \mapsto 3$, $ 5  \mapsto 4$, $ 6  \mapsto 4$, $ 7  \mapsto 6$, $ 8  \mapsto 6$ \\\hline 
\pbox{20cm}{1567, 678, 458, 237} & $ 1  \mapsto 1$, $ 2  \mapsto 2$, $ 3  \mapsto 2$, $ 4  \mapsto 2$, $ 5  \mapsto 3$, $ 6  \mapsto 3$, $ 7  \mapsto 4$, $ 8  \mapsto 4$ \\\hline 
\pbox{20cm}{1457, 678, 458, 237} & $ 1  \mapsto 1$, $ 2  \mapsto 1$, $ 3  \mapsto 1$, $ 4  \mapsto 2$, $ 5  \mapsto 2$, $ 6  \mapsto 2$, $ 7  \mapsto 3$, $ 8  \mapsto 3$ \\\hline 
\pbox{20cm}{45678, 1246, 678, 578, 346} & $ 1  \mapsto 2$, $ 2  \mapsto 2$, $ 3  \mapsto 4$, $ 4  \mapsto 5$, $ 5  \mapsto 3$, $ 6  \mapsto 7$, $ 7  \mapsto 5$, $ 8  \mapsto 5$ \\\hline 
\pbox{20cm}{35678, 2357, 678, 458, 123} & \scriptsize{ $ 1  \mapsto 18$, $ 2  \mapsto 25$, $ 3  \mapsto 30$, $ 4  \mapsto 28$, $ 5  \mapsto 40$, $ 6  \mapsto 27$, $ 7  \mapsto 37$, $ 8  \mapsto 44$} \\\hline 
\pbox{20cm}{35678, 1345, 678, 458, 237} & $ 1  \mapsto 2$, $ 2  \mapsto 3$, $ 3  \mapsto 5$, $ 4  \mapsto 4$, $ 5  \mapsto 5$, $ 6  \mapsto 4$, $ 7  \mapsto 6$, $ 8  \mapsto 6$ \\\hline
\pbox{20cm}{35678, 1246, 678, 578, 346} & $ 1  \mapsto 2$, $ 2  \mapsto 2$, $ 3  \mapsto 5$, $ 4  \mapsto 5$, $ 5  \mapsto 4$, $ 6  \mapsto 8$, $ 7  \mapsto 6$, $ 8  \mapsto 6$ \\\hline 
\pbox{20cm}{34678, 2357, 678, 458, 123} & \scriptsize{ $ 1  \mapsto 8$, $ 2  \mapsto 12$, $ 3  \mapsto 15$, $ 4  \mapsto 16$, $ 5  \mapsto 19$, $ 6  \mapsto 13$, $ 7  \mapsto 18$, $ 8  \mapsto 22$} \\\hline 
\pbox{20cm}{34578, 1345, 678, 458, 237} & $ 1  \mapsto 2$, $ 2  \mapsto 5$, $ 3  \mapsto 7$, $ 4  \mapsto 5$, $ 5  \mapsto 5$, $ 6  \mapsto 5$, $ 7  \mapsto 9$, $ 8  \mapsto 8$ \\\hline
\pbox{20cm}{34578, 1246, 678, 578, 346} & $ 1  \mapsto 2$, $ 2  \mapsto 2$, $ 3  \mapsto 4$, $ 4  \mapsto 5$, $ 5  \mapsto 3$, $ 6  \mapsto 7$, $ 7  \mapsto 5$, $ 8  \mapsto 5$ \\\hline 
\pbox{20cm}{34568, 1345, 678, 458, 237} & $ 1  \mapsto 2$, $ 2  \mapsto 4$, $ 3  \mapsto 6$, $ 4  \mapsto 5$, $ 5  \mapsto 5$, $ 6  \mapsto 6$, $ 7  \mapsto 8$, $ 8  \mapsto 8$ \\\hline
\pbox{20cm}{25678, 1345, 678, 458, 237} & $ 1  \mapsto 2$, $ 2  \mapsto 5$, $ 3  \mapsto 6$, $ 4  \mapsto 5$, $ 5  \mapsto 6$, $ 6  \mapsto 5$, $ 7  \mapsto 8$, $ 8  \mapsto 8$ \\\hline 
\pbox{20cm}{25678, 1246, 678, 578, 346} & \scriptsize{ $ 1  \mapsto 4$, $ 2  \mapsto 8$, $ 3  \mapsto 11$, $ 4  \mapsto 13$, $ 5  \mapsto 10$, $ 6  \mapsto 20$, $ 7  \mapsto 15$, $ 8  \mapsto 15$} \\\hline
\pbox{20cm}{24678, 1246, 678, 578, 346} & $ 1  \mapsto 2$, $ 2  \mapsto 2$, $ 3  \mapsto 4$, $ 4  \mapsto 5$, $ 5  \mapsto 3$, $ 6  \mapsto 7$, $ 7  \mapsto 5$, $ 8  \mapsto 5$ \\\hline 
\pbox{20cm}{24578, 1345, 678, 458, 237} & $ 1  \mapsto 2$, $ 2  \mapsto 7$, $ 3  \mapsto 7$, $ 4  \mapsto 6$, $ 5  \mapsto 6$, $ 6  \mapsto 7$, $ 7  \mapsto 11$, $ 8  \mapsto 10$ \\\hline
\pbox{20cm}{24578, 1246, 678, 578, 346} & $ 1  \mapsto 1$, $ 2  \mapsto 1$, $ 3  \mapsto 2$, $ 4  \mapsto 3$, $ 5  \mapsto 2$, $ 6  \mapsto 4$, $ 7  \mapsto 3$, $ 8  \mapsto 3$ \\\hline 
\pbox{20cm}{24568, 1345, 678, 458, 237} & $ 1  \mapsto 2$, $ 2  \mapsto 6$, $ 3  \mapsto 6$, $ 4  \mapsto 4$, $ 5  \mapsto 4$, $ 6  \mapsto 5$, $ 7  \mapsto 8$, $ 8  \mapsto 7$ \\\hline 
\pbox{20cm}{23678, 1246, 678, 578, 346} & $ 1  \mapsto 2$, $ 2  \mapsto 2$, $ 3  \mapsto 5$, $ 4  \mapsto 5$, $ 5  \mapsto 4$, $ 6  \mapsto 8$, $ 7  \mapsto 6$, $ 8  \mapsto 6$ \\\hline
 
\pbox{20cm}{23567, 1345, 678, 458, 237} & $ 1  \mapsto 2$, $ 2  \mapsto 4$, $ 3  \mapsto 5$, $ 4  \mapsto 4$, $ 5  \mapsto 5$, $ 6  \mapsto 5$, $ 7  \mapsto 7$, $ 8  \mapsto 7$ \\\hline
\pbox{20cm}{3456, 1458, 2378, 4678, \\ 2347, 2458} & $ 1  \mapsto 2$, $ 2  \mapsto 5$, $ 3  \mapsto 5$, $ 4  \mapsto 6$, $ 5  \mapsto 4$, $ 6  \mapsto 4$, $ 7  \mapsto 5$, $ 8  \mapsto 6$ \\\hline
\pbox{20cm}{2356, 1568, 3468, 2478, \\ 1268, 1248} & $ 1  \mapsto 7$, $ 2  \mapsto 9$, $ 3  \mapsto 6$, $ 4  \mapsto 7$, $ 5  \mapsto 5$, $ 6  \mapsto 9$, $ 7  \mapsto 3$, $ 8  \mapsto 10$ \\\hline
\pbox{20cm}{1357, 1356, 1348, 1346, \\ 1345, 1278, 1268} & \scriptsize{ $ 1  \mapsto 54$, $ 2  \mapsto 26$, $ 3  \mapsto 42$, $ 4  \mapsto 31$, $ 5  \mapsto 30$, $ 6  \mapsto 38$, $ 7  \mapsto 31$, $ 8  \mapsto 36$} \\\hline
\pbox{20cm}{1346, 1345, 1278, 1268, \\ 1267, 1258, 1257, 1256} & $ 1  \mapsto 7$, $ 2  \mapsto 6$, $ 3  \mapsto 2$, $ 4  \mapsto 2$, $ 5  \mapsto 5$, $ 6  \mapsto 5$, $ 7  \mapsto 4$, $ 8  \mapsto 4$ \\\hline
\pbox{20cm}{345678, 245678, 235678, 234678, \\ 234578, 234568, 234567, 145678, \\ 135678, 134678, 134578, 134568, \\ 134567, 125678, 124678, 124578, \\ 124568, 124567, 123678, 123578, \\ 123568, 123567, 123478, 123468,\\ 123467, 123456} & \scriptsize{ $ 1  \mapsto 28$, $ 2  \mapsto 28$, $ 3  \mapsto 28$, $ 4  \mapsto 28$, $ 5  \mapsto 28$, $ 6  \mapsto 30$, $ 7  \mapsto 29$, $ 8  \mapsto 29$} \\\hline
\pbox{20cm}{345678, 245678, 235678, 234678, \\ 234578, 234568, 234567, 145678, \\ 135678, 134678, 134578, 134568, \\ 134567, 125678, 124678, 124578, \\ 124568, 124567, 123678, 123578, \\ 123568, 123567, 123478, 123468, \\ 123457, 123456} & \scriptsize{ $ 1  \mapsto 27$, $ 2  \mapsto 27$, $ 3  \mapsto 27$, $ 4  \mapsto 27$, $ 5  \mapsto 28$, $ 6  \mapsto 28$, $ 7  \mapsto 28$, $ 8  \mapsto 28$} \\\hline
\end{tabular}
}
\caption{Frankl's conjecture holds for all UC families which contain the following subfamilies}
\end{center}
\end{table}

\begin{table}[h!]
\begin{center}
\scalebox{0.85}{
\begin{tabular}{|l|l|}
\hline
\multicolumn{2}{|c|}{Previously-unknown minimal nonisomorphic generators for FC-families on $[9]$}\\\hline
\hline
\pbox{20cm}{369, 789, 456, 123} & $ 1  \mapsto 1$, $ 2  \mapsto 1$, $ 3  \mapsto 2$, $ 4  \mapsto 1$, $ 5  \mapsto 1$, $ 6  \mapsto 2$, $ 7  \mapsto 1$, $ 8  \mapsto 1$, $ 9  \mapsto 2$ \\\hline
 
\pbox{20cm}{348, 569, 789, 1268} & \scriptsize{ $ 1  \mapsto 4$, $ 2  \mapsto 4$, $ 3  \mapsto 8$, $ 4  \mapsto 8$, $ 5  \mapsto 9$, $ 6  \mapsto 11$, $ 7  \mapsto 11$, $ 8  \mapsto 17$, $ 9  \mapsto 16$} \\\hline
\pbox{20cm}{148, 159, 6789, 2345} & $ 1  \mapsto 4$, $ 2  \mapsto 1$, $ 3  \mapsto 1$, $ 4  \mapsto 3$, $ 5  \mapsto 3$, $ 6  \mapsto 1$, $ 7  \mapsto 1$, $ 8  \mapsto 3$, $ 9  \mapsto 3$ \\\hline
\pbox{20cm}{589, 129, 6789, 3459} & \scriptsize{ $ 1  \mapsto 18$, $ 2  \mapsto 18$, $ 3  \mapsto 10$, $ 4  \mapsto 10$, $ 5  \mapsto 25$, $ 6  \mapsto 10$, $ 7  \mapsto 10$, $ 8  \mapsto 25$, $ 9  \mapsto 36$} \\\hline
\pbox{20cm}{489, 159, 2345, 5679} & \scriptsize{ $ 1  \mapsto 20$, $ 2  \mapsto 8$, $ 3  \mapsto 8$, $ 4  \mapsto 23$, $ 5  \mapsto 28$, $ 6  \mapsto 10$, $ 7  \mapsto 10$, $ 8  \mapsto 19$, $ 9  \mapsto 34$} \\\hline
\pbox{20cm}{5689, 578, 129, 6789, 3459} & $ 1  \mapsto 3$, $ 2  \mapsto 3$, $ 3  \mapsto 2$, $ 4  \mapsto 2$, $ 5  \mapsto 6$, $ 6  \mapsto 3$, $ 7  \mapsto 5$, $ 8  \mapsto 6$, $ 9  \mapsto 6$ \\\hline
\pbox{20cm}{5679, 128, 129, 6789, 3459} & \scriptsize{ $ 1  \mapsto 16$, $ 2  \mapsto 16$, $ 3  \mapsto 3$, $ 4  \mapsto 3$, $ 5  \mapsto 6$, $ 6  \mapsto 7$, $ 7  \mapsto 7$, $ 8  \mapsto 15$, $ 9  \mapsto 17$} \\\hline
\pbox{20cm}{5678, 278, 129, 6789, 3459} & \scriptsize{ $ 1  \mapsto 52$, $ 2  \mapsto 81$, $ 3  \mapsto 16$, $ 4  \mapsto 16$, $ 5  \mapsto 32$, $ 6  \mapsto 35$, $ 7  \mapsto 58$, $ 8  \mapsto 58$, $ 9  \mapsto 75$} \\\hline
\pbox{20cm}{4789, 578, 129, 6789, 3459} & \scriptsize{ $ 1  \mapsto 49$, $ 2  \mapsto 49$, $ 3  \mapsto 34$, $ 4  \mapsto 60$, $ 5  \mapsto 80$, $ 6  \mapsto 36$, $ 7  \mapsto 79$, $ 8  \mapsto 79$, $ 9  \mapsto 98$} \\\hline
\pbox{20cm}{4789, 489, 159, 6789, 2345} & \scriptsize{ $ 1  \mapsto 20$, $ 2  \mapsto 8$, $ 3  \mapsto 8$, $ 4  \mapsto 26$, $ 5  \mapsto 24$, $ 6  \mapsto 11$, $ 7  \mapsto 17$, $ 8  \mapsto 26$, $ 9  \mapsto 36$} \\\hline
\pbox{20cm}{4689, 578, 129, 6789, 3459} & \scriptsize{ $ 1  \mapsto 17$, $ 2  \mapsto 17$, $ 3  \mapsto 12$, $ 4  \mapsto 21$, $ 5  \mapsto 30$, $ 6  \mapsto 19$, $ 7  \mapsto 28$, $ 8  \mapsto 33$, $ 9  \mapsto 35$} \\\hline
\pbox{20cm}{4689, 478, 159, 6789, 2345} & \scriptsize{ $ 1  \mapsto 12$, $ 2  \mapsto 6$, $ 3  \mapsto 6$, $ 4  \mapsto 24$, $ 5  \mapsto 15$, $ 6  \mapsto 14$, $ 7  \mapsto 21$, $ 8  \mapsto 25$, $ 9  \mapsto 21$} \\\hline
\pbox{20cm}{4679, 158, 159, 6789, 2345} & \scriptsize{ $ 1  \mapsto 18$, $ 2  \mapsto 3$, $ 3  \mapsto 3$, $ 4  \mapsto 6$, $ 5  \mapsto 19$, $ 6  \mapsto 7$, $ 7  \mapsto 7$, $ 8  \mapsto 16$, $ 9  \mapsto 17$} \\\hline
\pbox{20cm}{4678, 578, 159, 6789, 2345} & \scriptsize{ $ 1  \mapsto 9$, $ 2  \mapsto 3$, $ 3  \mapsto 3$, $ 4  \mapsto 6$, $ 5  \mapsto 16$, $ 6  \mapsto 6$, $ 7  \mapsto 11$, $ 8  \mapsto 11$, $ 9  \mapsto 12$} \\\hline
\pbox{20cm}{4589, 589, 159, 6789, 2345} & $ 1  \mapsto 5$, $ 2  \mapsto 2$, $ 3  \mapsto 2$, $ 4  \mapsto 4$, $ 5  \mapsto 8$, $ 6  \mapsto 2$, $ 7  \mapsto 2$, $ 8  \mapsto 6$, $ 9  \mapsto 8$ \\\hline
\end{tabular}
}
\caption{Frankl's conjecture holds for all UC families which contain the following subfamilies}
\end{center}
\end{table}

\begin{table}[h!]
\begin{center}
\begin{tabular}{|l|l|}
\hline
\multicolumn{2}{|c|}{Previously-unknown minimal nonisomorphic generators for FC-families on $[10]$}\\\hline
\hline
\pbox{20cm}{123, 124, 356, 678, 79(10)} & \scriptsize{$ 1  \mapsto 6$, $ 2  \mapsto 6$, $ 3  \mapsto 8$, $ 4  \mapsto 4$, $ 5  \mapsto 5$, $ 6  \mapsto 7$, $ 7  \mapsto 5$, $ 8  \mapsto 4$, $ 9  \mapsto 2$, $ 10  \mapsto 2$} \\\hline
\pbox{20cm}{123, 124, 356, 678, 3489(10)} & \scriptsize{$ 1  \mapsto 7$, $ 2  \mapsto 7$, $ 3  \mapsto 5$, $ 4  \mapsto 5$, $ 5  \mapsto 5$, $ 6  \mapsto 6$, $ 7  \mapsto 3$, $ 8  \mapsto 3$, $ 9  \mapsto 1$, $ 10  \mapsto 1$} \\\hline
\end{tabular}
\caption{Frankl's conjecture holds for all UC families which contain the following subfamilies}
\end{center}
\end{table}

\iffalse
\vspace{2cm}
\begin{tabular}{|l|l|l|l|}
\hline
\multicolumn{4}{|c|}{Previously-unknown bounds and values of $FC(k,n)$}\\\hline
\hline
 $FC(4,6) = 7$ & $FC(5,6)$ \tt{not defined} & &\\\hline
 $FC(3,7) = 4$ & $FC(4,7)$ \tt{not defined} & $FC(5,7)$ \tt{not defined} &\\\hline 
 $FC(3,8) = 5$ & $FC(4,8)$ \tt{not defined} & $FC(5,8)$ \tt{not defined} & $FC(6,8)$ \tt{not defined}\\\hline 
 $FC(3,9) = 5$ & $FC(4,9)$ \tt{not defined} & $FC(5,9)$ \tt{not defined} & $FC(6,9)$ \tt{not defined}\\\hline
 $FC(3,10) = 6$ & $FC(4,10)$ \tt{not defined} & $FC(5,10)$ \tt{not defined} & $FC(6,10)$ \tt{not defined}\\\hline
 
\end{tabular}
\fi
\end{document}